\documentclass[a4paper,11pt]{amsart}

\usepackage{amsmath,amsthm}
\usepackage{amssymb,amsfonts}
\usepackage{hyperref}
\usepackage{indentfirst}
\usepackage{tikz}
\usepackage{epigraph}
\usepackage{enumerate}
\usepackage{graphicx}
\usepackage{calligra}
\usepackage{stmaryrd}
\usepackage{chngpage}
\usepackage{tikz-cd}
\usepackage{mathtools}
\DeclareSymbolFont{bbold}{U}{bbold}{m}{n}
\DeclareSymbolFontAlphabet{\mathbbm}{bbold}

\title{Trivial Endomorphisms of the Calkin Algebra}

\theoremstyle{plain}
	\newtheorem{theorem*}{Theorem}
	\newtheorem{theorem}{Theorem}[section]
\numberwithin{equation}{theorem}

	\newtheorem{proposition}[theorem]{Proposition}
	\newtheorem{lemma}[theorem]{Lemma}
	\newtheorem{corollary}[theorem]{Corollary}
	\newtheorem{claim}{Claim}[theorem]

\theoremstyle{definition}
	\newtheorem*{OCA*}{OCA}
	\newtheorem{definition}[theorem]{Definition}
	
	\newtheorem{example}[theorem]{Example}

	\newtheorem*{acknow}{Acknowledgements}
\theoremstyle{remark}
	\newtheorem{remark}[theorem]{Remark}

\newcommand{\C}{\mathbb{C}}
\newcommand{\N}{\mathbb{N}}

\newcommand{\set}[1]{\{#1\}}	
\newcommand{\cstar}{$\mathrm{C}^\ast$}
\newcommand{\cA}{\mathcal{A}}

\newcommand{\cQ}{\mathcal{Q}}
\newcommand{\cB}{\mathcal{B}}
\newcommand{\cD}{\mathcal{D}}
\newcommand{\cK}{\mathcal{K}}

\renewcommand{\phi}{\varphi}

\author{Andrea Vaccaro}
\address[A. Vaccaro]{Department of Mathematics, Universit\'e de Paris, 1 Place Aur\'elie Nemours, 
75013 Paris, France}
\email[]{vaccaro@imj-prg.fr}
\urladdr{https://sites.google.com/view/avaccaro/home}

\date{}

\begin{document}
\maketitle
\begin{abstract}
We prove that it is consistent with ZFC that every unital endomorphism of the Calkin algebra $\cQ(H)$ is unitarily equivalent to an endomorphism of $\cQ(H)$ which is
liftable to a unital endomorphism
of $\cB(H)$.
We use this result to classify all unital endomorphisms of $\cQ(H)$ up to unitary equivalence
by the Fredholm index of the image of the unilateral shift.
As a further application, we show that it is consistent with ZFC that the class of \cstar-algebras that
embed into $\cQ(H)$ is not closed under tensor product nor countable inductive limit.
\end{abstract}

\section{Introduction}
Let $H$ be a separable, infinite-dimensional, complex Hilbert space. The Calkin algebra $\cQ(H)$ is the
quotient of $\cB(H)$, the algebra of all linear, bounded operators on $H$, over the ideal of compact
operators $\cK(H)$. Let $q: \cB(H) \to \cQ(H)$ be the quotient map.

Over the last 15 years, the study of automorphisms of the Calkin algebra has been the setting
for some of the most significant applications of set theory
to \cstar-algebras. The original
motivation behind these investigations is of \cstar-algebraic nature, and it dates back
to the seminal paper \cite{bdf}.
There, it was asked whether there exists a K-theory reverting
automorphism of $\cQ(H)$ or, more concretely, an automorphism  of $\cQ(H)$
sending the unilateral shift to its adjoint. Since all inner automorphisms act trivially on the K-theory
of a unital \cstar-algebra, a preliminary question posed in \cite{bdf} is whether the Calkin algebra
has outer automorphisms at all.

The answer turned out to be depending on set theoretic axioms.
Phillips and Weaver showed in \cite{outer} that outer automorphisms of $\cQ(H)$ exist if the \emph{Continuum
Hypothesis} CH is assumed, while in \cite{inner} Farah proved that the \emph{Open Coloring Axiom} 
OCA (see Definition \ref{oca})
implies that all the automorphisms of $\cQ(H)$ are inner.
In particular, it is consistent with ZFC that there is no automorphism of $\cQ(H)$
sending the unilateral shift to its adjoint.
It is still unknown whether
the existence of an automorphism of $\cQ(H)$
sending the unilateral shift to its adjoint
is consistent with ZFC, since all the automorphisms built in \cite{outer} act
like inner automorphisms on every separable subalgebra of $\cQ(H)$.

In this note we investigate the effects of OCA on the endomorphisms of the Calkin algebra.
The main consequence of OCA that we show is a complete classification of the endomorphisms
of $\cQ(H)$, up to unitary equivalence, by the Fredholm index of
the image of the unilateral shift.
Two endomorphisms $ \phi_1, \phi_2 : \cQ(H) \to \cQ(H)$ are \emph{unitarily equivalent} if
there is a unitary $v \in \cQ(H)$ such that $\text{Ad}(v) \circ \phi_1 = \phi_2$.
Let $\text{End}(\cQ(H))$ be the set of all endomorphisms of $\cQ(H)$ modulo
unitary equivalence. Let $\text{End}_u(\cQ(H))$ be the set of all the classes in
$\text{End}(\cQ(H))$ corresponding to
unital endomorphisms. The operation of direct sum $\oplus$ naturally
induces a structure of semigroup on both $\text{End}(\cQ(H))$ and $\text{End}_u(\cQ(H))$,
as well as the operation of composition $\circ$.
Fix an orthonormal basis $\set{\xi_k}_{k \in \N}$ of $H$ and let $S$ be the unilateral shift
sending $\xi_k$ to $\xi_{k+1}$.

\begin{theorem} \label{mt2}
Assume OCA.
Two endomorphisms $ \phi_1, \phi_2: \cQ(H) \to \cQ(H)$ are unitarily equivalent if and only
if the following two conditions are satisfied:
\begin{enumerate}
\item \label{mta} There is a unitary $w \in \cQ(H)$ such that $w \phi_1(1) w^* = \phi_2(1)$.
\item \label{mtb} The (finite) Fredholm indices of $\phi_1(q(S)) + (1 - \phi_1(1))$ and $\phi_2(q(S)) + (1- \phi_2(1))$ are equal.
\end{enumerate}
Moreover, the map sending $\phi \in \text{End}_u(\cQ(H))$ to
$- \text{ind}(\phi(q(S)))$ is a semigroup isomorphism between $(\text{End}_u(\cQ(H)), \oplus)$ and
$(\N \setminus \set{0}, +)$, as well as between $(\text{End}_u(\cQ(H)), \circ)$
and $(\N \setminus \set{0}, \cdot)$.
\end{theorem}
An explicit description of $(\text{End}(\cQ(H)), \oplus)$
and $(\text{End}(\cQ(H)), \circ)$ under OCA is given in Remark \ref{nonu}.

The second consequence of OCA we prove in this paper is related to the recent works \cite{fhv},
\cite{fkv} and \cite[Chapter 2]{phd}, where methods from set theory are employed in the study of
nonseparable subalgebras of $\cQ(H)$.

Let $\mathbb{E}$ be the class of all \cstar-algebras
that embed into the $\cQ(H)$.
Under CH the Calkin algebra has density character $\aleph_1$ (the first uncountable cardinal), hence
a \cstar-algebra belongs to $\mathbb{E}$ if and only
if its density character is at most $\aleph_1$ (\cite[Theorem A]{fhv}). Therefore,
when CH holds, the class $\mathbb{E}$ is closed under all operations whose output,
when receiving as input \cstar-algebras of density character at most $\aleph_1$, is a \cstar-algebra
whose density character is at most $\aleph_1$. This includes, for instance, minimal/maximal tensor products
and countable inductive limits.

With the following result, we show that some of these closure properties might fail if CH is not assumed, answering
\cite[Question 5.3]{fkv}.
\begin{theorem} \label{closure}
Assume OCA.
\begin{enumerate}
\item \label{oca1} The class $\mathbb{E}$ is not closed under minimal/maximal tensor product.
Moreover, there exists $\cA \in \mathbb{E}$ such that
$\cA \otimes_{\gamma} \cB \notin \mathbb{E}$ for every infinite-dimensional, unital
$\cB \in \mathbb{E}$ and for every tensor norm $\gamma$.
\item \label{oca2} The class $\mathbb{E}$ is not closed under countable inductive limits.
\end{enumerate}
In particular, both \eqref{oca1} and \eqref{oca2} are independent from ZFC.
\end{theorem}

Theorem \ref{mt2} and Theorem \ref{closure} are proved in \S \ref{Sclass} using
Theorem \ref{thrm:main}, for which we need to introduce
a definition.

We say that an endomorphism $\phi: \cQ(H) \to \cQ(H)$
is \emph{trivial} if there is a unitary $v \in \cQ(H)$ and a strongly continuous (i.e. strong-strong continuous) endomorphism $\Phi: \cB(H) \to \cB(H)$
such that the following diagram commutes. 
\[
\begin{tikzcd}
\cB(H) \arrow{r}{\Phi} \arrow[swap]{d}{q} & \cB(H) \arrow{d}{q} \\%
\cQ(H) \arrow{r}{\text{Ad}(v) \circ \phi}& \cQ(H)
\end{tikzcd}
\]
With this terminology, the main theorem
of \cite{inner} says that under OCA all automorphisms of $\cQ(H)$ are trivial (since, up to
unitary equivalence, they lift to the identity). We extend this result to
all endomorphisms of $\cQ(H)$.

\begin{theorem} \label{thrm:main}
Assume OCA. All endomorphisms of the Calkin algebra are trivial.
\end{theorem}
The proof of Theorem \ref{thrm:main} occupies both \S\ref{S3} and \S\ref{sctn:main}.
Similarly to Theorem \ref{closure}, Theorem \ref{mt2} and Theorem \ref{thrm:main}
cannot be proved in ZFC alone, in fact they fail under CH.
In \S \ref{Sclass} we
give two examples
of non-trivial endomorphisms of $\cQ(H)$ existing when CH is assumed (Example \ref{ex1} and Example \ref{ex2}),
which highlight different levels of failure of the classification of $\text{End}(\cQ(H))$
in Theorem \ref{mt2}.
In particular, it is possible to find uncountably many (more precisely $2^{\aleph_1}$) inequivalent automorphisms of $\cQ(H)$
which send
the unilateral shift to a unitary of index $-1$ (see Example \ref{ex1}).
Example \ref{ex1} and Example \ref{ex2}, along with Theorem \ref{thrm:main},
entail the following corollary.

\begin{corollary} \label{indep}
The existence non-trivial endomorphisms of $\cQ(H)$ is independent from ZFC.
\end{corollary}

We remark that, unlike the results concerning automorphisms of $\cQ(H)$ in \cite{inner}, the
commutative analogue of Theorem \ref{thrm:main} for $\mathcal{P}(\N) / \text{Fin}$ does not
hold. In \cite{dow} it is proved that there are non-trivial endomorphisms of
$\mathcal{P}(\N) / \text{Fin}$ in ZFC. In this scenario, the best one can hope for is the so called
\emph{weak Extension Principle}, a consequence of $\text{OCA} + \text{MA}$ introduced and
discussed in \cite[Chapter 4]{analy} (those ideas
can already be found in \cite{just}).

A crucial feature of the Calkin algebra, which allows to carry on the arguments
needed to prove Theorem \ref{thrm:main}, is
the presence of partial isometries, which allow to compress/decompress operators
into/from infinite-dimensional
subspaces of $H$. These objects play a key role in the proof of Theorem \ref{thrm:lctr},
and they were already extensively used in \cite{inner} (see in particular Lemma 4.1 of that paper
and \cite[\S17.5]{ilijasbook}).
The existence of these isometries is directly responsible
for the fact that one does not need to assume Martin's Axiom when
proving Theorem \ref{thrm:main} or the main result of \cite{inner}. Such assumption
is instead customary
for many of the known proofs of rigidity phenomena in the context
of Boolan algebras (\cite{veloca}), and
of more
general corona algebras (\cite{mckvignati}, \cite{rig}), where these partial isometries
(or some sort of substitute) do not exist.

The observations and results exposed in this note
are in continuity with the numerous studies investigating the
strong rigidity properties induced by the \emph{Proper Forcing Axiom} (of which OCA is a consequence) on the Calkin algebra and on other
nonseparable quotient algebras
(\cite{inner}, \cite{mckvignati}, \cite{rig}, \cite{roecoronas}; see also \cite{analy}, \cite{veloca}).
The Continuum Hypothesis, on the
other hand,
grants the opposite effect, allowing to prove
the existence of too many maps on these quotients for all of them to be `trivial'
(\cite{outer} \cite{cosfar}, \cite{fms}). Woodin's $\Sigma^2_1$-absoluteness Theorem
gives a deep mathematical motivation for the efficacy of CH in solving these problems (see \cite{sigma12}).

The paper is structured as follows. Section \ref{S2} contains
preliminaries and definitions. Sections \ref{S3} and \ref{sctn:main} are devoted to the proof
of Theorem \ref{thrm:main}.
In \S\ref{S3} we show that locally trivial (see Definition \ref{trivial})
endomorphisms of $\cQ(H)$ are, up to unitary transformation, locally liftable with `nice' unitaries
in $\cB(H)$. In \S\ref{sctn:main}, adapting the main arguments from \cite{inner} (see also
\cite[\S17]{ilijasbook}), we prove that under
OCA all endomorphisms of $\cQ(H)$ are locally trivial, and that all locally trivial endomorphisms are trivial. We remark that OCA is only needed in
\S\ref{sctn:main}, while the results in \S\ref{S3} require no additional set-theoretic axiom. Finally,
\S \ref{Sclass} contains the proof of Theorem \ref{mt2}, the proof of Theorem \ref{closure}, some observations about what
 $\text{End}(\cQ(H))$ looks like under CH, and some open questions.

\section{Notation and Preliminaries} \label{S2}

The only extra set-theoretic assumption required for our proofs is the Open Coloring Axiom
OCA, which is defined as follows.
\begin{definition} \label{oca}
Given a set $X$, let $[X]^2$ be the set of all unordered pairs of elements of $X$. For a topological
space $X$, a \emph{coloring} $[X]^2 = K_0 \sqcup K_1$ is \emph{open} if the set $K_0$,
when naturally identified with a symmetric subset of $X \times X$, is open in the product topology.
For a $K \subseteq [X]^2$, a subset $Y$ of $X$ is $K$-homogeneous if $[Y]^2 \subseteq K$.
\begin{OCA*}
Let $X$ be a separable, metric space, and let $[X]^2 = K_0 \sqcup K_1$ be an open coloring.
Then either $X$ has an uncountable $K_0$-homogeneous set, or it can be covered by countably
many $K_1$-homogeneous sets.
\end{OCA*}
\end{definition}

This statement, which contradicts CH,
is independent from ZFC and it was introduced by Todorcevic in \cite{parti}.

For the rest of the paper, fix $\set{\xi_k}_{k \in \N}$
an orthonormal basis of $H$ and identify $\ell^\infty$, the \cstar-algebra of all uniformly bounded
sequences of complex numbers,
with the atomic masa of all operators in $\cB(H)$ diagonalized by such basis.
With this identification, the algebra $c_0 = \ell^\infty \cap \cK(H)$
is the set of all sequences converging to zero.

Given a set $M \subseteq \N$, $P_M$ denotes the orthogonal
projection onto the closure of $\text{span}\set{\xi_k : k \in M}$. If $M = \set{k}$,
we simply write $P_k$.
Throughout this paper, we focus on partitions $\vec{E}$ of $\N$ 
composed by consecutive, finite intervals. We use the notation
$\text{Part}_\N$ to denote the set of such partitions.\footnote{Note
that in \cite[\S 9.7]{ilijasbook} the notation
$\text{Part}_\N$ is used to denote a slightly larger family,
including also all partitions of cofinal subsets of $\N$ into finite intervals. In this paper we will exclusively
consider partitions of $\N$.}

Given a partition $\vec{E} = \set{E_n}_{n \in \N} \in \text{Part}_\N$,
$\cD[\vec{E}]$ is the von Neumann algebra of all operators in $\cB(H)$ for which each $\text{span}
\set{\xi_k : k \in E_n}$ is invariant. Equivalently, $\cD[\vec{E}]$ is the set of all operators which commute
with $P_{E_n}$ for every $n \in \N$. It is straightforward to check that
$\cD[\vec E] \cong \prod_{n \in \N} M_{\lvert E_n \rvert}(\C)$.
Given a subset $X \subseteq \N$, $\cD_X[\vec{E}]$ denotes
the \cstar-algebra of all the operators in $\cD[\vec{E}]$ which map to zero all elements in $\bigcup_{n \in \N \setminus X}
\{\xi_k : k \in E_n \}$.

Given a unital $*$-homomorphism $\Phi: \cD[\vec{E}] \to \cB(H)$
such that $\Phi[\cD[\vec{E}] \cap \mathcal{K}(H)] \subseteq \mathcal{K}(H)$, $\vec{n}_\Phi$
denotes the sequence $\set{\text{rk}(\Phi(P_k))}_{k \in \N}$ of the (finite) ranks of the
projections $\Phi(P_k)$.
Notice that $\vec{n}_\Phi$ only depends on how $\Phi$ acts on $\ell^\infty$ and that
$n_i = n_j$ whenever $i, j \in E_n$ for some $n \in \N$. For a partition $\vec{E} = \set{E_n}_{n \in \N} \in \text{Part}_\N$, $\vec{E}^{\text{even}} \in \text{Part}_\N$ is the partition $\set{E_{2n} \cup E_{2n+1}}_{n \in \N}$ and
$\vec{E}^{\text{odd}} \in \text{Part}_\N$ is the partition $\set{E_0} \cup \set{E_{2n+1} \cup E_{2n+2}}_{n \in \N}$.
These notions of even and odd partitions were first introduced in \cite{inner}.

The \emph{strong topology} on $\cB(H)$ (and on any subalgebra of $\cB(H)$) is the topology induced
by the pointwise norm convergence on $H$, hence
a sequence $\set{T_n}_{n \in \N}$ of operators in $\cB(H)$ \emph{strongly converges} to $T$ if and only if
$T_n \xi \to T \xi$ in norm for every $\xi \in H$.

For every partial isometry $v$ in a unital \cstar-algebra $\cA$, $\text{Ad}(v): \cA \to \cA$ is the map
sending $a$ to $vav^*$ for every $a \in \cA$.

For every subalgebra $\cA$ of $\cB(H)$, let
$\cA_\cQ$ be the quotient $\cA /(\mathcal{K}(H)\cap \cA)$. Given a map
$\phi: \cA_\cQ \to \cQ(H)$, the function $\Phi: \cA \to \cB(H)$ \emph{lifts} (or \emph{is a lift of}) $\phi$
if the following diagram commutes:
\begin{equation*}
\begin{tikzcd}
\cA \arrow{r}{\Phi} \arrow[swap]{d}{q} & \cB(H) \arrow{d}{q} \\%
\cA_\cQ \arrow{r}{\phi}& \cQ(H)
\end{tikzcd}
\end{equation*}

\begin{definition} \label{trivial}
Given a \cstar-algebra $\cA \subseteq \cB(H)$, we say that an embedding (i.e. an injective
$*$-homomorphism)
$\phi:\cA_\cQ \to \cQ(H)$ is \emph{trivial} if there exists a unitary $v \in \cQ(H)$
and a strongly continuous (i.e. strong-strong continuous),
$*$-homomorphism $\Phi: \cA \to \cB(H)$ such that $\Phi$ lifts $\text{Ad}(v) \circ \phi$. An endomorphism $\phi: \cQ(H) \to \cQ(H)$ is
\emph{locally trivial} if,
for every $\vec{E} \in \text{Part}_\N$, the restriction $\phi \restriction \cD[\vec{E}]_\cQ$ is trivial.
\end{definition}

Given two operators $T, S \in \cB(H)$, we use the notation $T \sim_{\cK(H)} S$ to abbreviate
$T - S \in \cK(H)$. Analogously, for a \cstar-algebra $\cA$ and two functions
$\Phi_1, \Phi_2: \cA \to \cB(H)$, $\Phi_1 \sim_{\cK(H)} \Phi_2$ abbreviates $\Phi_1(a) \sim_{\cK(H)}
\Phi_2(a)$ for all $a \in \cA$.

Given a \cstar-algebra $\cA \subseteq \cB(H)$ ($\cA \subseteq \cQ(H)$), the \emph{commutant} $\cA' \cap \cB(H)$ ($\cA' \cap \cQ(H)$) is the set
of all the operators in $\cB(H)$ ($\cQ(H)$) commuting with all elements of $\cA$.

The (\emph{Fredholm}) \emph{index} of an operator $T \in \cB(H)$ is the integer
$\text{dim}(\text{ker}(T)) - \text{codim}(\text{ran}(T))$. We say that  an operator $T$ is \emph{Fredholm}
if both $\text{dim}(\text{ker}(T))$ and $\text{codim}(\text{ran}(T))$ are finite. By Atkinson's Theorem (\cite[Theorem 1.4.16]{murphy}),
an operator $T$ is Fredholm if and only if $q(T)$ is invertible in $\cQ(H)$. In particular
any unitary $u \in \cQ(H)$ lifts to a Fredholm operator which can be assumed to be a partial isometry.

\begin{remark}  \label{remark:sc}
Strongly continuous, unital endomorphisms of 
$\cB(H)$ have an extremely
rigid structure. Indeed, strong continuity implies that such maps are uniquely determined by how they
behave on the projections whose range is 1-dimensional.
Since all these projections are Murray-von Neumann
equivalent, the same is true for their images, which therefore all have the same rank.

For every $m \in \N$, let $\Phi_m: \cB(H) \to \cB(H \otimes \C^m)$ be the map
sending $T$ to $T \otimes 1_m$. As $\cB(H)$ and $\cB(H \otimes \C^m)$ are isomorphic,
with an abuse of notation we consider $\Phi_m$ as a map from $\cB(H)$ into $\cB(H)$.
If $\Phi: \cB(H) \to \cB(H)$ is a unital, strongly continuous
endomorphism sending compact operators into compact operators,
by the previous observation it is possible to find an $m \in \N \setminus  \set{0}$ and a unitary
$U \in \cB(H)$ such that $\Phi = \text{Ad}(U) \circ \Phi_m$.

Our classification of $\text{End}(\cQ(H))$ and $\text{End}_u(\cQ(H))$ in Theorem \ref{mt2}
will be based on this simple observation.
Since the commutant of the image of $\ell^\infty$ via $\Phi_m$ is isomorphic to $\ell^\infty(
M_m(\C))$ (the \cstar-algebra of all norm-bounded sequences of $m \times m$ matrices with complex
entries), the same is true for the commutant of the image of $\ell^\infty$ via $\Phi$.
\end{remark}

Every unital endomorphism $\phi: \cQ(H) \to \cQ(H)$ is uniquely determined by its
restrictions $\set{\phi \restriction \cD[\vec{E}]_\cQ}$ as $\vec{E}$ varies in $\text{Part}_\N$. This is a consequence
of the following standard fact (see \cite[Lemma 1.2]{inner} or the proof of \cite[Theorem 3.1]{matroid}).
\begin{proposition} \label{prop:2qd}
For every countable set $\set{T_n}_{n \in \N}$ in $\cB(H)$ there exists a partition $\vec{E}\in  \text{Part}_\N$
such that for every $n \in \N$ there are $T^0_n \in \cD[\vec{E}^{\text{even}}]$ and
$T^1_n \in \cD[\vec{E}^{\text{odd}}]$ such that $T_n \sim_{\cK(H)} T_n^0 + T_n^1$.
\end{proposition}

\section{Locally Trivial Endomorphisms} \label{S3}

Given a unital, locally trivial endomorphism $\phi: \cQ(H) \to \cQ(H)$, throughout this section we fix,
for every partition $\vec{E} \in  \text{Part}_\N$, a partial isometry of finite index $v_{\vec{E}}$ and a strongly continuous,
$*$-homomorphism $\Phi_{\vec{E}} : \cD[\vec{E}] \to \cB(H)$ such that $\text{Ad}(v_{\vec{E}})
\circ \Phi_{\vec{E}}$ lifts the restriction of $\phi$ to $\cD[\vec{E}]_{\cQ}$.

In this section we show that,
up to considering $\text{Ad}(v) \circ \phi$ for some unitary $v \in \cQ(H)$, we can assume that
$\Phi_{\vec{E}}$ is $\Phi_m$\footnote{We denote the restriction of $\Phi_m$ to
$\cD[\vec{E}]$ by $\Phi_m$.} (as defined in Remark \ref{remark:sc}) and $v_{\vec{E}}$ is a unitary in the commutant of $\Phi_m[\ell^\infty]$,
for every partition $\vec{E} \in  \text{Part}_\N$.
We remark that no extra set-theoretic axiom is required in the present section.

\begin{remark} \label{remark:uni}
Notice that for a unital, locally trivial endomorphism $\phi: \cQ(H) \to \cQ(H)$,
for each partition $\vec{E} \in  \text{Part}_\N$ the projection $\Phi_{\vec{E}}(1)$ is a compact perturbation
of the identity, hence its range
has finite codimension $r$. Therefore, by multiplying $v_{\vec{E}}$
by a suitable partial isometry
of index $-r$, we can always assume that $\Phi_{\vec{E}}$ is unital, and we will always implicitly do
so.
\end{remark}

\begin{lemma} \label{lemma:seq}
Let $\Phi_1 : \ell^\infty \to \cB(H)$ and $\Phi_2 : \ell^\infty \to \cB(H)$ be two strongly continuous,
unital $*$-homomorpisms such that $\Phi_1[c_0], \Phi_2[c_0]
\subseteq \mathcal{K}(H)$. Suppose there exist two partial isometries of finite index
$v_1, v_2 $ such that $\text{Ad}(v_1) \circ \Phi_1 \sim_{\cK(H)}
\text{Ad}(v_2) \circ  \Phi_2$. Then the sequences $\vec{n}_{\Phi_1} = \set{\text{rk}(\Phi_1(P_k))}_{k \in \N}$ and $\vec{n}_{\Phi_2} = \set{\text{rk}(\Phi_2(P_k))}_{k \in \N}$ are eventually equal. 
\end{lemma}
\begin{proof}
The proof goes by contradiction.
Since $q(v_1)$ and $q(v_2)$ are unitaries in $\cQ(H)$, 
we can assume that $v_1$ is the identity, and we denote $v_2$ by $v$. Let $\vec{n}_{\Phi_1} = \set{n_k}_{k \in \N}$, $\vec{n}_{\Phi_2} = \set{m_k}_{k \in \N}$ and
suppose there is an infinite $X \subseteq \N$ such that $n_k  > m_k$ for every $k \in X$.

We quickly sketch the idea of the proof first. Since
\[
{\text{rk}(\Phi_1(P_{n_k})) > \text{rk}(\Phi_2(P_{n_k}}))\ge
\text{rk}(\text{Ad}(v)  (\Phi_2(P_{n_k}))),
\]
it is always possible to find a norm-one vector 
$\zeta_k$ in the image of $\Phi_1(P_{n_k})$ which is also in the kernel of $\text{Ad}(v)  (\Phi_2(P_{n_k}))$. Moreover, for every $j \in \N$ big enough, the vector $\zeta_k$ is, up to a small error, in the complement
of the range of $\text{Ad}(v)  (\Phi_2(P_{n_j}))$.
Because of this,
by choosing an infinite subset $Y \subseteq X$ whose elements are spaced far enough apart,
we can find an infinite sequence of orthonormal vectors $\{ \zeta_k \}_{k \in \N}$
which all belong to $\Phi_1(P_Y)$, but such that $\lVert \text{Ad}(v)( \Phi_2(P_Y)) \zeta_k \rVert \le 1/2$
for every $k \in \N$. This contradicts the assumption $ \Phi_1 \sim_{\cK(H)}
\text{Ad}(v) \circ  \Phi_2$.

We inductively define the sets
$Y \subseteq X$ and $\{\zeta_k \}_{k \in \N}$ as follows.
Let $y_0$ be the minimum of $X$ and let $Y_0 = \set{y_0}$. Since
$\text{rk}(\Phi_1(P_{y_{0}})) > \text{rk}(\Phi_2(P_{y_{0}}))\ge
\text{rk}(\text{Ad}(v)  (\Phi_2(P_{y_{0}})))$, there is a norm-one vector
$\zeta_{0}$ in the image of $\Phi_1(P_{y_{0}})$ which also belongs to the kernel 
of $\text{Ad}(v) (\Phi_2(P_{y_{0}}))$. This is the case since the codimension of $\text{ker}
(\text{Ad}(v) ( \Phi_2(P_{y_{0}})))$ is strictly smaller than $\text{rk}(\Phi_1(P_{y_{0}}))$.

Suppose $Y_k = \set{y_0 < \dots < y_k}\subseteq X$ and that,
for every $h \le k$, there is a norm-one vector $\zeta_h$ such that $\Phi_1(P_{y_h})\zeta_h =
\zeta_h$ and $\lVert \text{Ad}(v)  (\Phi_2(P_{Y_k}))
\zeta_h \rVert < 1/2$. Let $y_{k+1}$ be the smallest element in $X$ greater than $y_k$ such that
\begin{enumerate}
\item \label{itema} $\lVert \text{Ad}(v) ( \Phi_2(P_{Y_k})) \Phi_1(P_{y_{k+1}}) \rVert < 1/2$,
\item \label{itemb} $\lVert \text{Ad}(v)  (\Phi_2(P_{Y_k \cup \set{y_{k+1}}})) \zeta_h \rVert < 1/2$
for every $h \le k$. 
\end{enumerate}
Such number
$y_{k+1}$ exists since all $\text{Ad}(v) (\Phi_2(P_{Y_k}))$ and the projection
onto $\text{span}\set{v^*\zeta_h : h \le k}$ have finite rank and, by strong continuity,
the sequences $\set{\Phi_1(P_k)}_{k \in \N}$ and $\set{\Phi_2(P_k)}_{k \in \N}$
strongly converge to zero. Define $Y_{k+1} = Y_k \cup \set{y_{k+1}}$.

We have to verify that $Y_{k+1}$ satisfies
the inductive hypothesis, namely
that for every $h \le k+1$ there is a norm-one vector $\zeta_h$ such that $\Phi_1(P_{y_h})\zeta_h =
\zeta_h$ and $\lVert  \text{Ad}(v) (\Phi_2(P_{Y_{k+1}}))\zeta_h
\rVert < 1/2$. For $h \le k$, pick the $\zeta_h$ given by
the inductive hypothesis, and the inequality follows
by item \eqref{itemb}.
Since $y_{k+1} \in X$, it follows that $\text{rk}(\Phi_1(P_{y_{k+1}})) > \text{rk}(\Phi_2(P_{y_{k+1}}))\ge
\text{rk}(\text{Ad}(v)( \Phi_2(P_{y_{k+1}})))$.
There exists thus a norm-one vector
$\zeta_{k+1}$ in the image of $\Phi_1(P_{y_{k+1}})$ which also belongs to the kernel 
of $	\text{Ad}(v)( \Phi_2(P_{y_{k+1}}))$.

Because of this and item \eqref{itema}:
\begin{multline*}
\lVert  \text{Ad}(v) ( \Phi_2(P_{Y_{k+1}}) )\zeta_{k+1} \rVert =
\lVert \text{Ad}(v)  (\Phi_2(P_{Y_k})) \zeta_{k+1} \rVert  \le \\ \le \lVert \text{Ad}(v)
 (\Phi_2(P_{Y_k}))  \Phi_1(P_{y_{k+1}}) \rVert  < 1/2.
\end{multline*}
Let $Y = \cup_{k \in \N} Y_k$. We show that, for every $k \in \N$, the following holds
\[
\lVert (\Phi_1 (P_Y) -  \text{Ad}(v) (\Phi_2 (P_Y))) \zeta_k \rVert \ge 1/2,
\]
which contradicts $\Phi_1(P_Y) \sim_{\mathcal{K}(H)} \text{Ad}(v)( \Phi_2(P_Y))$.
The previous
inequality follows since, for every $k \in \N$, by strong continuity of $\Phi_1$ we have that
$\Phi_1 (P_Y) \zeta_k = \zeta_k$,
and by strong continuity of $\Phi_2$ we have that
\[
\lVert \text{Ad}(v) ( \Phi_2(P_Y)) \zeta_k \rVert \le 1/2.
\]
\end{proof}

\begin{lemma} \label{prop:const}
Let $\phi: \cQ(H) \to \cQ(H)$ be a unital, locally trivial endomorphism. There exists $m \in \N$ such that,
for every $\vec{E}$, the sequence $\vec{n}_{\Phi_{\vec{E}}} = \set{\text{rk}(\Phi_{\vec{E}}
(P_k))}_{k \in \N}$ is, up to a finite number of entries, constantly equal to $m$.
\end{lemma}
\begin{proof}
By Lemma \ref{lemma:seq}, it is enough to show that there exists a partition $\vec{E} \in  \text{Part}_\N$ such that
$\vec{n}_{\Phi_{\vec{E}}}$ is eventually constant.
Let  $\vec{E}_1$ be the partition composed by the intervals $\set{2k,2k+1}$ and $\vec{E}_2$ be the partition
composed by $\set{2k+1, 2k+2}$, as $k$ varies in $\N$. Let $\vec{n}_{\Phi_{\vec{E_1}}} =
\set{n_k}_{k \in \N}$ and $\vec{n}_{\Phi_{\vec{E_2}}} = \set{m_k}_{k \in \N}$. On the one hand we have
that $n_{2k} = n_{2k+1}$ and $m_{2k+1} = m_{2k+2}$ for all $k \in \N$, since these couple of numbers
belong to the same intervals in $\vec{E}_1$ and $\vec{E}_2$ respectively.
On the other hand, since by assumption both $\text{Ad}(v_{\vec{E_1}}) \circ \Phi_{\vec{E_1}}$ and
$\text{Ad}(v_{\vec{E_2}}) \circ \Phi_{\vec{E_2}}$ lift $\phi$ on $\ell^\infty/c_0$, it follows that
$\text{Ad}(v_{\vec{E_1}}) \circ \Phi_{\vec{E_1}} \sim_{\cK(H)} \text{Ad}(v_{\vec{E_2}}) \circ \Phi_{\vec{E_2}}$ on $\ell^\infty$. Remember moreover that we assume
that $\Phi_{\vec{E_1}}$ and $\Phi_{\vec{E_2}}$ are unital (see Remark \ref{remark:uni}). Thus, by Lemma
\ref{lemma:seq}, there is $j \in \N$ such that $n_i = m_i$ for all $i \ge j$. It follows that
there is $m \in \N$ such that $n_i = m_i = m$ for all $i \ge j$.
 \end{proof}
 
 Let $\Phi: \cD[\vec{E}] \to \cB(H)$ be a strongly continuous, unital $*$-homomorphism such that
 $\vec{n}_\Phi$ is eventually constant with value $m$.
 This is not enough to infer that, even up to a unitary transformation of $H$,
 $\Phi$ is a compact perturbation of $\Phi_m$. For instance, the map $\Psi: \ell^\infty \to \cB(H)$
 sending $(a_0, a_1, a_2, \dots) \mapsto (a_0, a_0, a_1, a_2, \dots)$ is not a compact
 perturbation of the identity, since it is a compact perturbation of $\text{Ad}(S)$, being $S$ the unilateral
 shift. Nevertheless, by suitably `shifting' $\Phi$ it is possible to obtain a compact
 perturbation of $\Phi_m$, as shown in the following lemma.

\begin{lemma} \label{lemma:shift}
Let $\vec{E} \in  \text{Part}_\N$ and let $\Phi: \cD[\vec{E}] \to \cB(H)$ be a strongly continuous,
unital $*$-homomorphism such that $\Phi[\cD[\vec{E}] \cap \cK(H)] \subseteq \cK(H)$ and such that
$\vec{n}_\Phi$ is eventually constant with value $m \in \N \setminus \set{0}$.
There exists a partial isometry $w$ of finite index such that
$\text{Ad}(w) \circ \Phi \sim_{\mathcal{K}(H)} \Phi_m$.
\end{lemma}
\begin{proof}
If two strongly continuous, unital embeddings $\Phi_1, \Phi_2: \cD[\vec{E}] \to \cB(H)$ are such that
$\vec{n}_{\Phi_1} =  \vec{n}_{\Phi_2}$, then it is always possible to find a unitary $u \in \cB(H)$ sending
$\Phi_2(P_k)H$ to $\Phi_1(P_k)H$
for every $k \in \N$ such that $\text{Ad}(u) \circ  \Phi_1 = \Phi_2$. Because of this,
if a strongly continuous, unital $*$-homomorphism $\tilde \Phi :\cD[\vec{E}] \to \cB(H)$ is such
that $\vec{n}_{ \tilde{\Phi}}$ is the constant sequence with value $m$, then there is a unitary $u' \in
\cB(H)$ such that $\text{Ad}(u') \circ \tilde \Phi \sim_{\mathcal{K}(H)} \Phi_m$.
In view of this, to prove the lemma
it is enough to show that
there is a partial isometry $w$ of finite index and a strongly continuous, unital $*$-homomorphism
$\tilde{\Phi}: \cD[\vec{E}] \to \cB(H)$ such that $\tilde{\Phi} \sim_{\cK(H)}
\text{Ad}(w) \circ \Phi$ and such that $\vec{n}_{ \tilde{\Phi}}$ is constantly equal to $m$.

Let
$\vec{n}_\Phi = \set{n_k}_{k \in \N}$. By assumption, there is
$h \in \N$ such that for all $k \in E_j$ with $j \ge h$, we have
$n_k = m$.
Let $n = \sum_{ k < \min{E_h}} n_k$ and let $r = m\cdot \min{E_h} - n$.
The value $r$ tells us how much we need to `shift' $\Phi$ in order to make it, up to unitary equivalence, a compact perturbation of $\Phi_m$.

More concretely, set $Q:= \sum_{j < h} P_{E_j}$. The corner $Q\cB(H) Q$ is the portion
of $\cB(H)$ containing all projections $P_k$ such that $\text{rk}(\Phi(P_k)) \not = m$,
hence it is the corner where we want to modify $\Phi$.
Since $P_{E_j}$ belongs to the centre of $\cD[\vec E]$ for every $j \in \N$, we also obtain
that $Q$ is a central element of $\cD[\vec E]$.

Fix $\set{\zeta_i}_{i \in \N}$ an orthonormal basis
of $H$ such that
\begin{equation} \label{eq:basis}
\set{\zeta_i}_{i < n} \text{ is an orthonormal basis of } \Phi(Q)H.
\end{equation}
Let $P$ be the orthogonal projection onto the space spanned by $\set{\zeta_i}_{i < n + r}$,
and let $\Psi$ be a unital embedding of $Q \cB(H) Q \cong M_{\bar k}(\C)$ to
$P \cB(H) P \cong M_{m \bar k}(\C)$, where $\bar k = \min{E_h}$, so $n + r = \bar k m$. The idea at this point is to
define a $\tilde \Phi$ which behaves like $\Psi$ on $Q\cB(H) Q$, and like a
shift by $r$ of $\Phi$ on $(1- Q) \cB(H) (1 - Q)$, so that $\tilde \Phi$ is a strongly continuous,
unital $*$-homomorphism with $n_{\tilde \Phi}$ constantly equal to $m$.

This is concretely done as follows. Let $S$ be the unilateral
shift sending $\zeta_i$ to $\zeta_{i+1}$.
We have that
\begin{equation} \label{eq:proj}
S^{-r}S^r \ge \Phi(1-Q),
\end{equation}
since the right-hand side of the inequality is the projection onto the space spanned by $\{ \zeta_i \}_{i \ge n}$ by \eqref{eq:basis},
whereas the  left-hand side is either the identity (if $r \ge 0$) or
the orthogonal projection onto $\text{span}\set{\zeta_i}_{i \ge -r}$ (if $r \le 0$), and by definition
$-r = n - m \cdot \min{E_h} \le n$.

By \eqref{eq:proj} the operator $\text{Ad}(S^r \Phi(1-Q))$ is a projection, moreover it is exactly the projection onto the space spanned by $\set{\zeta_i}_{i \ge n + r}$,
that is $\text{Ad}(S^r \Phi(1-Q)) = 1- P$. This is the case since the range of $\Phi(1 -Q)$
is spanned by $\{ \zeta_i \}_{i \ge n}$ by \eqref{eq:basis}, and applying $\text{Ad}(S^r)$ to it shifts
its range forward by $r$.

As anticipated, let $\tilde{\Phi} := (\Psi \circ \text{Ad}(Q)) \oplus (\text{Ad}(S^{r} \Phi (1-Q)) \circ \Phi)$. The map
$\tilde{\Phi}$ is clearly strongly continuous. The fact that $\tilde \Phi$ is unital
follows from $\Psi(Q) = P$ and the observations in the previous paragraph:
\begin{align*}
\tilde \Phi (1) &= \Psi(Q) \oplus \text{Ad}(S^r \Phi(1-Q))(1) \\
&= \Psi(Q) \oplus
\text{Ad}(S^r \Phi(1-Q))
\\
&= P \oplus (1 - P) = 1.
\end{align*}
Multiplicativity of $\tilde \Phi$ follows since $S^{-r}S^r \ge \Phi(1 - Q)$ and $Q$ commutes with every element in
$\cD[\vec{E}]$.

Finally, for $k < \min{E_h}$ we have that $\tilde \Phi(P_k) = \Psi(P_k)$, which has rank $m$
by definition of $\Psi$. For $k \ge \min{E_h}$ we have $\tilde \Phi(P_k)= S^r \Phi(P_k) S^{-r}$, whose
rank is $m$ since $\text{rk}(\Phi(P_k)) = m$ and since applying to it
$\text{Ad}(S^r)$ does not change the rank, as $S^r$
is injective on the range of $\Phi(1-Q)$ (which contains the range of $\Phi(P_k)$ for $k \ge \min{E_h}$)
by \eqref{eq:proj}. This proves that
$\vec{n}_{\tilde{\Phi}}$ is constantly
equal to $m$.
The proof is thus concluded since we proved that
$\tilde{\Phi} \sim_{\cK(H)} \text{Ad}(S^r) \circ \Phi$.
\end{proof}

The following lemma, an analogue of \cite[Lemma 1.4]{inner}, shows
that unital, locally trivial embeddings which lift to $\Phi_m$ on $\ell^\infty/c_0$ have
nice and regular lifts also on the other $\cD[\vec{E}]_\cQ$'s.

\begin{lemma} \label{lemma:uni}
Let $\phi: \cQ(H) \to \cQ(H)$ be a unital, locally trivial endomorphism such that $\Phi_m$ lifts
$\phi$ on $\ell^\infty/ c_0$. Then, for every partition $\vec{E} \in  \text{Part}_\N$, there exists a unitary $u_{\vec{E}}$
in $\Phi_m[\ell^\infty]' \cap \cB(H) \cong \ell^\infty(M_m(\C))$
such that $\text{Ad}(u_{\vec{E}}) \circ \Phi_m$ lifts $\phi_{\vec{E}}$ on $\cD[\vec{E}]_\cQ$.
\end{lemma}
\begin{proof}
Fix a partition $\vec{E} \in  \text{Part}_\N$, let $v_{\vec{E}}$ be a partial isometry of finite index and let $\Phi_{\vec{E}}:
\cD[\vec{E}] \to \cB(H)$ be a strongly continuous, unital $*$-homomorphism such that
$\text{Ad}(v_{\vec{E}}) \circ \Phi_{\vec{E}}$ lifts $\phi$ on $\cD[\vec{E}]$,
which exist since $\phi$ is locally trivial.

By assumption we have that
$\text{Ad}(v_{\vec{E}}) \circ \Phi_{\vec{E}} \sim_{\mathcal{K}(H)} \Phi_m$ on $\ell^\infty$.
Thus, by Lemma \ref{lemma:seq}, the sequence $n_{\Phi_{\vec{E}}}$ is eventually constant with
value $m$. This in turn implies, by Lemma \ref{lemma:shift}, that
there is a finite index isometry $w$ such that $\text{Ad}(v_{\vec{E}})
\circ \Phi_{\vec{E}} \sim_{\mathcal{K}(H)}
\text{Ad}(w) \circ\Phi_m$ on $\cD[\vec{E}]$, hence the latter also lifts $\phi$ on $\cD[\vec{E}]_\cQ$.

Combining the fact that both
$\text{Ad}(v_{\vec{E}}) \circ \Phi_{\vec{E}} \sim_{\mathcal{K}(H)} \Phi_m$
and $\text{Ad}(v_{\vec{E}})
\circ \Phi_{\vec{E}} \sim_{\mathcal{K}(H)}
\text{Ad}(w) \circ\Phi_m$ on $\ell^\infty$, we obtain that $\text{Ad}(w)
\circ \Phi_m \sim_{\mathcal{K}(H)} \Phi_m$ on $\ell^\infty$. This entails that $w$ commutes,
up to compact operators, with the elements in $\Phi_m[\ell^\infty]$.
The commutant of $\Phi_m[\ell^\infty]$ is (isomorphic
to) $\ell^\infty(M_m(\C)) \subseteq \cB(H)$ and by \cite[Theorem 2.1]{jp} we have that $w$
is a compact perturbation of an element $u$ in $\ell^\infty(M_m(\C))$.

Let $u = v \lvert u \rvert$ be the polar decomposition of $u$, where $v$ is a partial isometry.
Notice that, since $\ell^\infty(M_m(\C))$
is a von Neumann algebra, $v$ belongs to $\ell^\infty(M_m(\C))$. Moreover,
as $q(u) = q(w) \in \cQ(H)$ is a unitary, it follows that $\lvert u \rvert \sim_{\cK(H)} 1$, thus
$v \sim_{\cK(H)} u \sim_{\cK(H)} w$. In particular $v$ is a partial isometry 
such that $q(v) = q(w)$ and
\begin{equation} \label{eq:pi}
v^*v \sim_{\cK(H)} vv^* \sim_{\cK(H)} 1.
\end{equation} 
Being $v =(v_n)_{n \in \N}$ a partial isometry, we have that $vv^*v = v$.
Thus $v_nv_n^*v_n = v_n$ for every $n \in \N$, that is each $v_n$ is a partial isometry in $M_m(\C)$.
As a consequence, both $v_n^*v_n$ and $v_nv_n^*$ are projections in $M_m(\C)$.
Formula \eqref{eq:pi} entails that for all but finitely many $n \in \N$ we have $v_n^* v_n = v_n
v_n^* = 1_{m}$, namely $v_n$ is a unitary in $M_m(\C)$. Thus, up to changing finitely
many entries of $v = (v_n)_{n \in \N}$, we obtain a unitary $u_{\vec{E}}$ in the commutant of $\Phi_m[\ell^\infty]$ such that
$u_{\vec{E}} \sim_{\cK(H)} w$. We have therefore that
\[
\text{Ad}(u_{\vec E}) \circ\Phi_m \sim_{\mathcal{K}(H)} \text{Ad}(w) \circ\Phi_m \sim_{\mathcal{K}(H)} \text{Ad}(v_{\vec{E}})
\circ \Phi_{\vec{E}}  ,
\]
on $\cD[\vec{E}]$, hence $\text{Ad}(u_{\vec E}) \circ\Phi_m$ lifts $\phi$ on $\cD[\vec E]_\cQ$.
\end{proof}

\section{All Endomorphisms are Trivial} \label{sctn:main}
We split the proof of Theorem \ref{thrm:main} in two steps. We first prove that all unital, locally
trivial endomorphisms of $\cQ(H)$ are trivial, then we show that all unital endomorphisms
of $\cQ(H)$ are locally trivial. We use OCA in both proofs. The non-unital case follows
from the unital one, since every endomorphism $\phi: \cQ(H) \to \cQ(H)$ can be thought as a unital
endomorphism with codomain $\cQ(\phi(1)H)$.

\subsection{Locally trivial endomorphisms are trivial} \label{4bis}
\begin{theorem} \label{thrm:main2}
Assume $OCA$.
Every unital, locally trivial endomorphism $\phi: \cQ(H) \to \cQ(H)$ is trivial.
\end{theorem}
Fix a unital endomorphism $\phi: \cQ(H) \to \cQ(H)$. The endomorphism $\phi$ is
trivial if and only if $\text{Ad}(v) \circ \phi$ is trivial
for some unitary $v \in \cQ(H)$. Because of this and by Lemma \ref{prop:const}, Lemma \ref{lemma:shift} and Lemma \ref{lemma:uni},
we can assume that there is $m \in \N$ such that $\phi$ lifts to $\Phi_m$ when
restricted to $\ell^\infty /c_0$, and that for every partition $\vec{E} \in  \text{Part}_\N$ there is a unitary $u_{\vec{E}}$
in $\ell^\infty(M_m(\C))$ such that $\text{Ad}(u_{\vec{E}}) \circ \Phi_m$ lifts $\phi$ on $\cD[\vec{E}]_\cQ$. Given a unitary $u \in \ell^\infty(M_m(\C))$ we denote $\text{Ad}(u) \circ \Phi_m$
by $\Phi^u$ (the endomorphism $\phi$, and therefore the integer $m$,
will be always fixed through this section, hence we omit $m$ in this notation).

The proof of Theorem \ref{thrm:main2} is inspired to \cite[Section 3]{inner},
where Theorem \ref{thrm:main2} is proved for an automorphism, hence in case $m = 1$.
The idea is to glue together the various
$u_{\vec{E}}$ in a coherent way in order to define a unitary $u \in \ell^\infty(M_m(\C))$
such that $\Phi^u$ lifts $\phi$ globally. We identify the unitaries in
$\ell^\infty(M_m(\C))$
with elements in $(\mathcal{U}(M_m(\C)))^\N$, being $\mathcal{U}(M_m(\C))$ the unitary
group of $M_m(\C)$.

For $u =(u(i))_{i \in \N}, v= (v(i))_{i \in \N} \in \mathcal{U}(\ell^\infty(M_m(\C)))$ and $I \subseteq \N$, define
\[
\Delta_I(u,v) : = \sup_{i,j \in I} \lVert u(i) u^*(j) - v(i) v^*(j) \rVert.
\]
\begin{lemma} \label{lemma:delta}
For all $I \subseteq \N$ and $u, v \in \mathcal{U}(\ell^\infty(M_m(\C)))$:
\begin{enumerate}
\item \label{item:delta1} $\Delta_I(u,v) \le 2 \sup_{i \in I} \lVert u(i) - v(i) \rVert$.
\item \label{item:delta2} $\Delta_I(u,v) \ge \sup_{j \in I} \lVert u(j)- v(j) \rVert - \inf_{i \in I} \lVert u(i) - v(i)\rVert$. In particular,
if $u(k) = v(k)$ for some $k \in I$, then $\Delta_I(u,v) \ge \sup_{j \in I} \lVert u(j) - v(j) \rVert$.
\item \label{item:delta3} If $w \in \mathcal{U}(M_m(\C))$ then $\Delta_I(u,v) = \Delta_I(u,vw)$.
\item \label{item:delta4} If $I \cap J \not = \emptyset$, then $\Delta_{I \cup J}(u,v) \le \Delta_I(u,v) + \Delta_J(u,v)$.
\item \label{item:delta5} $\inf_{w \in \mathcal{U}(M_m(\C))} \sup_{i \in I} \lVert u(i) - v(i) w \rVert \le \Delta_I(u,v)
\le 2 \inf_{w \in \mathcal{U}(M_m(\C))} \linebreak \sup_{i \in I}  \lVert u(i) - v(i) w \rVert$.
\end{enumerate}
\end{lemma}
\begin{proof}
The proof of this lemma can be easily inferred from the proof of \cite[Lemma 1.5]{inner}. We have that
\begin{align*}
\lVert u(i) u^*(j) - v(i) v^*(j) \rVert &=  \lVert v^*(i)(u(i) u^*(j) - v(i) v^*(j)) u(j) \rVert 
 \\ &=
\lVert v^*(i) u(i) - 1 + 1- v^*(j) u(j) \rVert \\
&= \lVert v^*(i)(u(i) - v(i)) - v^*(j)(u(j) - v(j)) \rVert.
\end{align*}
This entails
\begin{multline*}
\lvert \lVert u(i) - v(i) \rVert - \lVert u(j) - v(j) \rVert \rvert \le \lVert u(i)u^*(j) - v(i)v^*(j) \rVert \le \\ 
\le \lVert u(i) - v(i) \rVert + \lVert u(j) - v(j) \rVert,
\end{multline*}
from which both item \eqref{item:delta1} and \eqref{item:delta2} follow. Item \eqref{item:delta3} is straightforward
to check since $v(i) ww^* v^*(j) = v(i) v^*(j)$. Notice that, unlikely the 1-dimensional case, it is important
to consider $vw$ rather than $wv$. Item \eqref{item:delta4} follows by the triangular inequality, since
$\lVert u(i) u^*(j) - v(i) v^*(j) \rVert$ is equal to $\lVert v^*(i) u(i) - v^*(j) u(j) \rVert$. Item \eqref{item:delta5} follows by item \eqref{item:delta3} plus
items \eqref{item:delta1} and \eqref{item:delta2}.
\end{proof}

\begin{lemma} \label{lemma:cmpct}
Let $u = (u(i))_{i \in \N},v= (v(i))_{i \in \N} \in \mathcal{U}(\ell^\infty(M_m(\C)))$.
\begin{enumerate}
\item \label{item:cmpct1} If $\lim_{i \to \infty} \lVert u(i) - v(i) \rVert = 0$ then $\Phi^u \sim_{\mathcal{K}(H)} \Phi^v$.
\item \label{item:cmpct2} $\Phi^u \sim_{\mathcal{K}(H)} \Phi^v$ on $\cD[\vec{E}]$ if and only if
$\lim \sup_n \Delta_{E_n} (u,v) = 0$.
\end{enumerate}
\end{lemma}
\begin{proof}
This lemma (and its proof) is an adapted version of \cite[Lemma 1.6]{inner} for endomorphisms.
If $\lim_{i \to \infty} \lVert u(i) - v(i) \rVert$ is zero, it means that $u \sim_{\mathcal{K}(H)} v$,
hence $\Phi^u \sim_{\mathcal{K}(H)} \Phi^v$.

In order to prove item \eqref{item:cmpct2}, suppose first that $\lim \sup_n \Delta_{E_n} (u,v) = 0$.
For every $n \in \N$, let $k_n$ be $\min(E_n)$. Let $w =(w(i))_{i \in \N} \in \ell_{\infty}(M_m(\C))$
be the unitary defined, for $i \in E_n$, as
\[
w(i):= v(i) v^*(k_n) u(k_n).
\]
The unitary $\sum  P_{E_n} \otimes v^*(k_n) u(k_n) $ belongs to the commutant of
$\Phi_m[\cD[\vec{E}]]$, hence $\Phi^w = \Phi^v$ on $\cD[\vec{E}]$. On the other hand,
by items \eqref{item:delta2}-\eqref{item:delta3} of Lemma \ref{lemma:delta} we have that, for $i \in E_n$, $\lVert w(i) - u(i) \rVert
\le \Delta_{E_n}(u,w) = \Delta_{E_n}(u,v)$. Thus $\lim_{i \to \infty} \lVert w(i) - u(i) \rVert = 0$ and,
by item \eqref{item:cmpct1} of this lemma, $\Phi^u \sim_{\mathcal{K}(H)} \Phi^w = \Phi^v$.

To prove the other direction, suppose there is $\epsilon > 0$ and a subsequence $\set{n_k}_{k \in \N}$ such that
$\Delta_{E_{n_k}}(u,v) > \epsilon$. Fix two sequences $i_k,j_k \in E_{n_k}$ such that
$\lVert u(j_k)u^*(i_k) - v(j_k)v^*(i_k) \rVert > \epsilon$ for every $k \in \N$.
Let $\eta_k \in \C^m$ be a norm-one vector witnessing the previous inequality. Let $V$ be the partial
isometry in $\cD[\vec{E}]$ moving $\xi_{i_k}$ to $\xi_{j_k}$ (from the orthonormal
basis of $H$ we fixed at the beginning of \S\ref{S2}) for every $k \in \N$ and sending
all other vectors in $\set{\xi_n}_{n \in \N}$ to zero. We have that, if $\zeta \in \Phi_m(P_{i_k})$,
\[
\Phi^u(V) (\zeta) = u\Phi_m(V) u^*(i_k) (\zeta) = u(j_k) u^*(i_k)(\zeta),
\]
\[
\Phi^v(V) (\zeta) = v\Phi_m(V) v^*(i_k) (\zeta) = v(j_k) v^*(i_k)(\zeta).
\]
Thus, for the vector $\eta_k$ we fixed before (or rather for $0 \oplus \dots \oplus 0 \oplus \eta_k \oplus 0 \dots$,
where the non-zero coordinate appears in the $i_k$-th position), we have
\[
\lVert (\Phi^u(V) - \Phi^v(V)) (\eta_k) \rVert = \lVert (u(j_k)u^*(i_k) - v(j_k)v^*(i_k))(\eta_k) \rVert > \epsilon.
\]
Since this holds for every $k \in \N$, it follows that the difference $\Phi^u(V) - \Phi^v(V)$ is not compact.
\end{proof}

Given a function $f \in \N^\N$ we recursively define, as in \cite[\S 3.1]{inner}, the map $f^+$
by $f^+(0) = f(0)$ and $f^+(n +1) = f(f^+(n))$ for every $n \ge 1$.
Moreover we define, for $f \in \N^\N$ and $n \in \N$
\[
E^f_n:= [f(n), f(n+1)),
\]
\[
F^f_n:=[f^+(n), f^+(n+1)),
\]
\[
E^{f, \text{even}}_n := [f(2n), f(2n+2)),
\]
\[
E^{f, \text{odd}}_n := [f(2n+1), f(2n+3)).
\]

The corresponding partitions are $\vec{E}^f$, $\vec{F}^f$, $\vec{E}^{f, \text{even}}$ and $\vec{E}^{f, \text{odd}}$ respectively. We shall denote $\vec{E}^{f^+, \text{even}}$ and
$\vec{E}^{f^+, \text{odd}}$ by $\vec{F}^{f, \text{even}}$ and $\vec{F}^{f, \text{odd}}$.
\begin{lemma} \label{lemma:double}
Let $\phi: \cQ(H) \to \cQ(H)$ be a unital, locally trivial endomorphism which can be lifted to
$\Phi_m$ on $\ell_{\infty}/c_0$ for some $m \in \N$. For every $f \in \N^{\N}$ there is a unitary
$w \in \ell^\infty(M_m(\C))$ such that $\Phi^w$ lifts $\phi$ on both $\cD[\vec{E}^{f, \text{even}}]$
and $\cD[\vec{E}^{f, \text{odd}}]$.
\end{lemma}
\begin{proof}
This proof follows the one of \cite[Lemma 3.5]{inner}.
By assumption there are two unitaries $u, v \in \ell^\infty(M_m(\C))$ such that $\Phi^u$ and
$\Phi^v$ lift $\phi$ on $\cD[\vec{E}^{f, \text{even}}]$ and $\cD[\vec{E}^{f, \text{odd}}]$ respectively.
We define inductively two unitaries $u', v' \in \ell^\infty(M_m(\C))$ as follows. For $i \in [f(0), f(2))$,
let $u'(i) = u(i)$. If $u'(i)$ has been defined for $i < f(2n)$, for $i \in [f(2n - 1), f(2n + 1))$ let
\[
v'(i) = v(i) v^*(f(2n - 1)) u'(f(2n-1)).
\]
If $v'(i)$ has been defined for $i < f(2n +1)$, for $i \in [f(2n), f(2n + 2))$ let
\[
u'(i) = u(i) u^*(f(2n )) v'(f(2n)).
\]
We have that $v'(f(n)) = u'(f(n))$, that $\Phi^u= \Phi^{u'}$ on $\cD[\vec{E}^{f, \text{even}}]$ and that
$\Phi^v= \Phi^{v'}$ on $\cD[\vec{E}^{f, \text{odd}}]$. This implies that, by item
\eqref{item:delta2} of Lemma \ref{lemma:delta},
\[
\sup_{i \in E^f_n} \lVert u'(i) - v'(i) \rVert \le \Delta_{E^f_n}(u', v').
\]

 On the other hand we have that
 $\Phi^{u'} = \Phi^u \sim_{\mathcal{K}(H)} \Psi^v = \Psi^{v'}$ on $\cD[\vec{E}^f]$
 by hypothesis (remember that $\cD[\vec{E}^f] \subseteq \cD[\vec{E}^{f, \text{even}}] \cap \cD[\vec{E}^{f, \text{odd}}]$), therefore by item \eqref{item:cmpct2}
 of Lemma \ref{lemma:cmpct} it follows that $\lim_{n \to \infty} \Delta_{E^f_n}(u', v') = 0$.
 By item \eqref{item:cmpct1} of Lemma \ref{lemma:cmpct}, we infer that $\Phi^{u'}$ and
 $\Psi^{v'}$ agree on $\cB(H)$
 up to compact operator. In conclusion, $\Phi^{u'}$ lifts $\phi$ on both $\cD[\vec{E}^{f, \text{even}}]$
and $\cD[\vec{E}^{f, \text{odd}}]$.
\end{proof}

Given $f, g \in \N^\N$, we write $g \le^* f$ if $g(n) \le f(n)$ for all but finitely many $n \in \N$. A subset
$\mathcal{F} \subseteq \N^\N$ is \emph{$\le^*$-cofinal} if for every $g \in \N^\N$ there is $f \in \mathcal{F}$
such that $g\le^* f$.

The next two lemmas are taken from \cite{inner}, but these ideas and arguments
can be traced back to \cite{analy} (see e.g. \cite[Lemma 2.2.2]{analy}) and \cite{parti}.

\begin{lemma}[{\cite[Lemma 3.3]{inner}}] \label{lemma:cof0}
Assume $\mathcal{F} \subseteq \N^\N$ is $\le^*$-cofinal.
\begin{enumerate}
\item \label{item:cof01} If $\mathcal{F}$ is partitioned into countably many pieces, then at least one
is $\le^*$-cofinal.
\item \label{item:cof02} $(\exists^\infty n) (\exists i) (\forall k \ge n)(\exists f \in \mathcal{F})
(f(i) \le n \text{ and } f(i+1) \ge k)$.
\item \label{item:cof03} $\set{f^+ : f \in \mathcal{F}}$ is $\le^*$-cofinal.
\end{enumerate}
\end{lemma}

\begin{lemma}[{\cite[Lemma 3.4]{inner}}] \label{lemma:cof}
Let $f,g \in \N^\N$ be such that $g \le^* f$. For all but finitely many $n \in \N$ there is $i$ such that
$f^+(i) \le g(n) < g(n+1) \le f^{+}(i+2)$. If $f(m) \ge g(m)$ for all $m \in \N$, then the previous statement holds for every $n \in \N$.
\end{lemma}

Lemma \ref{lemma:cof} entails that if $g \le^* f$ then for all but finitely many $n \in \N$ there
is $i_n \in \N$ such that $E^g_n \subseteq F^f_{i_n} \cup F^f_{i_n+1}$. In particular, if $f(m) \ge g(m)$ for
all $m \in \N$, then $\cD[\vec{E}^g]$ is contained in the algebra generated by
$\cD[\vec{F}^{f, \text{even}}] \cup \cD[\vec{F}^{f, \text{odd}}]$.
\begin{proof}[Proof of Theorem \ref{thrm:main2}]
We can assume that $\phi: \cQ(H) \to \cQ(H)$ is locally
represented on $\cD[\vec{E}]$ by $\Phi^{u_{\vec{E}}}$, where $u_{\vec{E}}$ is a unitary in
$\ell^\infty(M_m(\C))$ (see the paragraph after the statement of Theorem \ref{thrm:main2}). Let $\mathcal{X}
\subset \N^\N \times \mathcal{U}(M_m(\C))^\N$ be the set of all pairs $(f, u)$ such that
$\Phi^u$ lifts $\phi$ on both $\cD[\vec{F}^{f, \text{even}}]$ and $\cD[\vec{F}^{f, \text{odd}}]$.
By Lemma \ref{lemma:double}, for every $f \in \N^\N$ there is $u$ such that
$(f,u) \in \mathcal{X}$.

Fix $\epsilon > 0$ and consider the coloring of $[\mathcal{X}]^2 = K_0^\epsilon \sqcup K_1^\epsilon$, where
the pair $(f, u)$, $(g, v)$ has color $K_0^\epsilon$ if there are $m,n \in \N$ such that
$\Delta_{F^f_n \cap F^g_m}(u,v) > \epsilon$. We consider $\N^\N$ with the Baire space topology,
induced by the metric
\[
d(f,g) = 2^{-\min \set{n: f(n) \not=g(n)}}.
\]
This is a complete separable metric. We consider $\mathcal{U}(M_m(\C))^\N$ with the product
of the norm topology on $\mathcal{U}(M_m(\C))$ and $\mathcal{X}$ with the product topology. In this setting, it is straightforward
to check that $K_0^\epsilon$ is open.
\begin{claim} \label{main:claim1}
Assume OCA. For every $\epsilon > 0$ there are no uncountable $K_0^\epsilon$-homogeneous
subsets of $\mathcal{X}$.
\end{claim}
\begin{proof}
Fix $\epsilon > 0$ and let $\mathcal{H}$ be an uncountable $K_0^\epsilon$-homogeneous
subset of $\mathcal{X}$. Let
\[
\mathcal{F} = \set{g^+: \exists u (g,u) \in \mathcal{H}}.
\]
We can assume that $\mathcal{H}$, and thus $\mathcal{F}$, has size $\aleph_1$.
Under OCA, all $\le^*$-unbounded subsets of $\N^\N$
have size at least $\aleph_2$ (\cite[Theorems 3.4 and 8.5]{parti}).
Thus, there is $f \in \N^\N$ which is an upper bound for $\mathcal{F}$.
Using the pigeonhole principle, we can assume that there is $\overline{n} \in \N$
such that $f(m) \ge g^+(m)$ for all $g^+ \in \mathcal{F}$ and all $m \ge \overline{n}$, and moreover that
$g^+(i) = h^+(i)$ for all $g^+, h^+ \in \mathcal{F}$ and all $i \le \overline{n}$.
By increasing $f$ by $f(\overline{n})$
we can also assume that $f(m) \ge g^+(m)$ for all $g^+ \in \mathcal{F}$ and $m \in \N$.
By Lemma
\ref{lemma:cof} this entails that for every $n \in \N$ there is $i \in \N$ such that $E^{g^+}_n = F^g_{n}
\subseteq F^f_i \cup F^f_{i+1}$.

Let $u$ be a unitary in $\ell_{\infty}(M_m(\C))$ such that $\Phi^u$ lifts
$\phi$ on $\cD[\vec{F}^{f, \text{even}}]$ and $\cD[\vec{F}^{f, \text{odd}}]$.
Since,
by the previous observations, for every $g^+ \in \mathcal{F}$ we have that $\cD[\vec{F}^g]$
is contained in the algebra generated by $\cD[\vec{F}^{f, \text{even}}] \cup \cD[\vec{F}^{f, \text{odd}}]$,
it follows that $\Phi^u$ lifts $\phi$ also on $\cD[\vec{F}^g]$ and therefore, by item \eqref{item:cmpct2} of
Lemma \ref{lemma:cmpct}, we have that $\lim_{n \to \infty}
\Delta_{F^g_n} (u,v) = 0$ for every $(g,v) \in \mathcal{H}$. By taking an uncountable subset of $\mathcal{F}$ if necessary, we can assume that there is $\overline{k}$
such that $\Delta_{F^g_m}(u,v) < \epsilon/2$ for all $m \ge \overline{k}$ and all $(g, v) \in \mathcal{H}$.
By separability of $\mathcal{U}(M_m(\C))^\N$ there are $(g,v)$, $(h,w) \in \mathcal{H}$ such that
$g^+(i) = h^+(i)$ for all $i  \le \overline{k}$ and $\lVert w_i - v_i \rVert < \epsilon /2$
for all $i \le g^+(\overline{k})$. This entails that if $n,m \in \N$ are such that
$F_n^g \cap F_m^h \not = \emptyset$, then either both $m,n \le \overline{k}$ or $m,n \ge \overline{k}$.
In the former case it follows that $\Delta_{F_n^g \cap F_m^h}(v,w) < \epsilon$ by item \eqref{item:delta1}
of Lemma \ref{lemma:delta}. If $m,n \ge \overline{k}$ then we have
\[
\Delta_{F_n^g \cap F_m^h}(w,v) \le \Delta_{F_n^g \cap F_m^h}(u,v) + \Delta_{F_n^g \cap F_m^h}(w,u)
< \epsilon.
\]
This is a contradiction since $(g,v)$, $(h,w) \in \mathcal{H}$.
\end{proof}
By OCA, for every $\epsilon > 0$ there is a partition of $\mathcal{X}$ into countably many $K^\epsilon_1$-homogeneous sets.
Let $\epsilon_n = 2^{-n}$. Repeatedly using item \eqref{item:cof01} of Lemma \ref{lemma:cof0},
find sequences $\mathcal{X} \supseteq \mathcal{X}_0 \supseteq \dots \supseteq \mathcal{X}_n
\supseteq \dots$ and $0 = m(0) < m(1) < \dots < m(n) < \dots$ such that $\mathcal{X}_n$ is
$K^{\epsilon_n}_1$-homogeneous and such that the set $\set{f : (\exists u) (f,u) \in \mathcal{X}_n}$ is $\le^*$-cofinal.
Let $m(n)$ be the natural number given by item \eqref{item:cof02} of Lemma \ref{lemma:cof0} for
$\mathcal{X}_n$. For each $n \in \N$ fix a sequence $\set{(f_{n,i}, u_{n,i})}_{i \in \N}$ in $\mathcal{X}_n$
such that, for some $j_i \in \N$
\begin{equation} \label{ineq}
f_{n,i}^+(j_i) \le m(n) < m(n+i) \le f_{n,i}^+(j_i +1).
\end{equation}
By compactness of $\mathcal{U}(M_m(\C))^\N$, we can assume that each sequence $\set{u_{n,i}}_{
i \in \N}$ converges to some $u_n$.
\begin{claim} \label{main:claim2}
There is a subsequence $\set{u_{n_k}}_{k \in \N}$ such that
\[
\sup_{i \in [m(n_h), \infty)} \lVert
u_{n_k}(i) - u_{n_h}(i) \rVert \le \epsilon_k
\]
for all $ h \ge k$.
\end{claim}
\begin{proof}
We start by showing that $\Delta_{[m(n), \infty)} (u_h, u_n) \le \epsilon_h$ for all $h < n$. Suppose this
is not the case and let $m(n) \le i_1 < i_2$ be such that $\lVert u_h(i_1)u^*_h(i_2) - u_n(i_1)u^*_n(i_2)
\rVert > \epsilon_h$. There is $j \in \N$ such that
\[
\lVert u_{h,j}(i_1)u^*_{h,j}(i_2) -
u_{n,j}(i_1)u^*_{n,j}(i_2) \rVert > \epsilon_h
\]
and by \eqref{ineq} there are $k_1, k_2 \in \N$ such that
\[
f^+_{n,j}(k_1) \le m(n) < i_2 < f^+_{n,j}(k_1+1),
\]
\[
f^+_{h,j}(k_2) \le m(h) < m(n) < i_2 < f^+_{h,j}(k_2+1).
\]
In particular, this entails that $\Delta_{F^{f_{h,j}}_{k_1} \cap F^{f_{n,j}}_{k_2}} (u_{h,j}, u_{n,j}) > \epsilon_h$,
which is a contradiction since $(f_{h,j}, u_{h,j})$ and $(f_{n,j}, u_{n,j})$ both belong to $\mathcal{X}_h$,
which is $K_1^{\epsilon_h}$-homogeneous.

By
item \eqref{item:delta5} of Lemma \ref{lemma:delta}, for every $h < n$
there is $w_{h,n} \in \mathcal{U}(M_m(\C))$ such that
\[
\sup_{i \ge m(n)} \lVert u_n(i) - u_h(i)w_{h,n} \rVert \le \epsilon_h.
\]
The unitary $w_{h,n}$ exists by compactness of $\mathcal{U}(M_m(\C))$.
Given $h < n < k \in \N$ we have that, for $i \ge
m(k)$
\[
u_h(i) w_{h,n} \approx_{\epsilon_h} u_n(i) \approx_{\epsilon_h}  u_k(i) w^*_{n,k} \approx_{\epsilon_h}
u_h(i) w_{h,k} w^*_{n,k},
\]
hence
\begin{equation} \label{unit}
\lVert w_{h,n} - w_{h,k} w^*_{n,k} \rVert \le 3 \epsilon_h.
\end{equation} This can be used to show that
there is an infinite $Y = \set{n_k}_{k \in \N} \subseteq \N$ such that, for $i < j \in \N$, then
$\lVert 1 - w_{n_i,n_j} \rVert \le 4 \epsilon_{n_i}$. To see this, define a coloring $M_0 \sqcup M_1$
on the triples of elements in $\N$, by saying that the triple $i < j < k$ is in $M_0$ if and only if
\[
\lVert 1- w_{j,k} \rVert \le 4 \epsilon_i.
\]

Suppose there is an infinite $M_1$-homogeneous set $Y$. Let $h$ be the minimum of $Y$. By
compactness of the unit ball of $M_m(\C)$ there is $n \in Y$ big enough so that, for some
$j < k < n$ all in $Y$ we have that $\lVert w_{j,n} - w_{k,n} \rVert < \epsilon_h$. It follows that
\[
\lVert w_{j,k} - 1 \rVert \le \lVert w_{j,k} - w_{j,n}w^*_{k,n} \rVert + \lVert 1 - w_{j,n}w^*_{k,n} \rVert
\stackrel{\mathclap{\eqref{unit}}}{\le} 4\epsilon_h,
\]
which is a contradiction, since the triple $(h,j,k)$ is supposed to be in $M_1$.

By Ramsey's theorem there
is an infinite $M_0$-homogeneous set $Y = \set{n_k}_{k \in \N}$. We have therefore, for $j > i \ge 1$
\[
\lVert 1 - w_{n_i, n_j} \rVert \le 4 \epsilon_{n_{i-1}}.
\]
Without loss of generality we can assume that $n_0 \ge 4$, hence that, for every $i \ge 1$
the following holds
\[
4 \epsilon_{n_{i-1}} \le \epsilon_{i-1}/4 = \epsilon_{i+1}.
\]
Summarizing, we have that for every $k < h \in \N$
\begin{align*}
\sup_{i \ge m(n_h)} \lVert u_{n_k}(i) - u_{n_h}(i) \rVert &\le \sup_{i \ge m(n_h)} \lVert u_{n_k}(i) w_{n_k,n_h}
- u_{n_h}(i) \rVert + \lVert 1 - w_{n_k,n_h} \rVert \\ &\le \epsilon_{n_k} + \epsilon_{k+1} \le \epsilon_k.
\end{align*}
\end{proof}
Let $\set{u_{n_k}}_{k \in \N}$ be the subsequence given by the previous claim and let $v \in 
\mathcal{U}(M_m(\C))^\N$ be defined as $v(i) = u_{n_k}(i)$ for all $i \in [m(n_k), m(n_{k+1}))$
and $v(i) = u_{n_0}(i)$ for all $i \le m(n_0)$. It follows then that $\lVert v(i) - u_{n_k}(i) \rVert
< \epsilon_k$ for all $i \ge m(n_k)$. Given $j \in \N$ and $(g,w) \in \mathcal{X}_{n_j}$, we claim
that for all $i \in \N$ we have $\Delta_{F^g_i \setminus m_{n_j}} (v,w) \le 3\epsilon_j$. This
is the case since for every $i \in \N$ there is $h\in \N$ such that
\[
[f_{n_j,h}^+(j_h), f_{n_j,h}^+(j_h + 1)) \supseteq F^g_i \setminus m_{n_j}.
\]
Hence, since both $g$ and $f_{n_j,h}$ belong to $\mathcal{X}_{n_j}$, we have that
$\Delta_{F^g_i \setminus m_{n_j}} (w, u_{n_j,h}) < \epsilon_{n_j}$. By continuity, we also have
$\Delta_{F^g_i \setminus m_{n_j}} (w, u_{n_j}) \le \epsilon_{n_j}$. Thus, in conclusion
\begin{align} \label{eq1}
\begin{split}
\Delta_{F^g_i \setminus m_{n_j}} (v,w) &\le \Delta_{F^g_i \setminus m_{n_j}} (v,u_{n_j}) + \Delta_{F^g_i \setminus m_{n_j}} (u_{n_j},w) \\ &\le 2 \sup_{h \in F^g_i \setminus m_{n_j}} \lVert v(h) - u_{n_k}(h) \rVert
+ \epsilon_{n_j} \\ &\le 3 \epsilon_j.
\end{split}
\end{align}

We conclude by showing that $\Phi^v$ lifts $\phi_{\vec{E}}$ for every partition $\vec{E} \in  \text{Part}_\N$. Let $g \in \N^\N$
be such that $\vec{E} = \vec{E}^g$ and find $u \in \mathcal{U}(M_m(\C))^\N$ such that $\Phi^u$
lifts $\phi$ on $\cD[\vec{E}^g]$. By item \eqref{item:cmpct2} of Lemma \ref{lemma:cmpct}, it
is enough to show that $\lim_{n \to \infty} \Delta_{E^g_n} (u,v) = 0$. Fix $k \in \N$, and let
$(f,w) \in \mathcal{X}_{n_k}$ be such that $f \ge^* g$. For all but finitely many $n\in \N$
there is $i \in \N$ such that $E^g_n \subseteq F^f_{i_n} \cup F^f_{i_n+1}$. This implies that
$\lim_{n \to \infty} \Delta_{E^g_n}(w,u) = 0$, which, by item \eqref{item:delta4} of Lemma \ref{lemma:delta}
in turn entails
\begin{align*}
\lim_{n \to \infty} \Delta_{E^g_n} (u,v) &\le \lim_{n \to \infty} \Delta_{E^g_n} (u,w) + \Delta_{E^g_n} (w,v) =
\lim_{n \to \infty} \Delta_{E^g_n} (w,v)  \\ &\le \lim_{n \to \infty} \Delta_{F^f_{i_n} \cup F^f_{i_n+1}} (w,v)
\\ &\le  \lim_{n \to \infty} \Delta_{F^f_{i_n} } (w,v) + \Delta_{F^f_{i_n+1}} (w,v) \\ &\stackrel{\mathclap{\eqref{eq1}}}{\le} 6 \epsilon_k.
\end{align*}
The inequality above holds for every $k \in \N$, thus $\lim_{n \to \infty} \Delta_{E^g_n} (u,v)$ is zero.
\end{proof}

\subsection{All endomorphisms are locally trivial}
\begin{theorem} \label{thrm:lctr}
Assume OCA.
Every unital endomorphism $\phi: \cQ(H) \to \cQ(H)$ is locally trivial.
\end{theorem}
\begin{proof}
Given a partition $\vec{E} \in  \text{Part}_\N$, we want to find a finite index isometry $v_{\vec{E}}$ and a strongly
continuous, unital $*$-homomorphism $\Phi_{\vec{E}}: \cD[\vec{E}] \to \cB(H)$ such that
$\text{Ad}(v_{\vec{E}}) \circ \Phi_{\vec{E}}$ lifts $\phi$ on $\cD[\vec{E}]_\cQ$.
Without loss of generality, we take a partition $\vec{E} \in \text{Part}_\N$ composed
by intervals whose length is strictly increasing.

First, we need a fact following from OCA which is proved in \cite[\S6, \S7]{inner} and \cite[\S17]{ilijasbook}.
What OCA entails is the existence of a strongly continuous $*$-homomorphism $\Psi: \cD[\vec{E}] \to \cB(H)$
which lifts $\phi$ on $\cD_X[\vec{E}]$, for some infinite $X \subseteq \N$.  The paper \cite{inner} focuses on automorphisms $\cQ(H)$, but these
proofs also work for unital endomorphisms, as shown in \cite[\S17]{ilijasbook}.
More specifically, \cite[Lemma 17.6.3]{ilijasbook} (see the errata corrige \cite[Lemma 17.6.3]{corrige}), \cite[Lemma 17.7.1]{ilijasbook}, \cite[Lemma 17.7.2]{ilijasbook} and \cite[Lemma 17.4.5]{ilijasbook},
by successive refinements on liftings of $\phi$, allow to
find an infinite $Y \subseteq \N$ and a strongly continuous map $\Psi':\cD[\vec E] \to \cB(H)$
which lifts $\phi$ on $\cD_Y[\vec{E}]$. From there, the
proof of \cite[Theorem 6.3]{inner} (which does not require OCA and works verbatim for endomorphisms) shows how to find an infinite
$X \subseteq Y$ and refine $\Psi'$ to a strongly continuous $*$-homomorphism
$\Psi: \cD[\vec{E}] \to \cB(H)$ such that $\Psi$ lifts $\phi$ on $\cD_X[\vec{E}]$.

Alternatively, it is possible to use \cite[Theorem 8.4]{mckvignati} to directly obtain $X$
and $\Psi$. This result however, being a more general statement about corona algebras, requires the stronger assumption $\text{OCA}_\infty + \text{MA}_{\aleph_1}$ (see \cite[\S2.2]{mckvignati}). 

Given such lift $\Psi$ and $X \subseteq \N$, the idea now is to exploit
the abundance of partial isometries in $\cQ(H)$ to obtain a global unital lift in $\cD[\vec{E}]$.\footnote{This final step of the proof has also been proved independently in \cite[Proposition 17.5.5]{ilijasbook}.
Since at the time of the first drafting of this paper, the book \cite{ilijasbook} was in a preliminary stage
without such result, we carried out the rest of the proof in full detail. For the reader's convenience,
we decided to include it also in this final version of the paper.} Concretely, we `compress' elements of $\cD[\vec{E}]$ into $\cD_X[\vec{E}]$, apply $\Psi$
and finally we `decompress' their image in $\cB(H)$
(see also \cite[Lemma 4.1]{inner}).

Let $v \in \cB(H)$ be a partial isometry such that $v^*v = 1$, $P := vv^* \le P_X$ 
belongs to $\cD_X[\vec{E}]$ and such that
$v \cD[\vec{E}] v^* \subseteq \cD_X[\vec{E}]$. Such partial isometry exists since, by
assumption, the length of the intervals in $\vec{E}$ is strictly increasing.
Let $Q$ be the image of $P$ via $\Psi$. Since $P \in \cD_X[\vec{E}]$ we have that
\begin{equation} \label{QP}
q(Q) = \phi(q(P)).
\end{equation}
Let $w$ be a partial isometry lifting $\phi(q(v))$, hence we have
\begin{equation} \label{w}
q(w) = \phi(q(v)), \
1 \sim_{\cK(H)} w^*w \text{ and } Q
\sim_{\cK(H)} ww^*.
\end{equation}
\begin{claim} \label{cc}
Up to a compact perturbation, we can assume that $w$ satisfies the following porperties.
\begin{enumerate}
\item \label{ic1} $ww^* \le Q$.
\item \label{ic2} $w^*w \ge Q$.
\end{enumerate}
\end{claim}
\begin{proof}
We start by proving item \eqref{ic1}. Let $w'$ be a lift of $\phi(q(v^*))$.
Then $w'Q$ is also a lift of $\phi(q(v^*))$, since
\[
q(w'Q) = \phi(q(v^*)) \phi(q(P)) = \phi (q(v^*P)) = \phi(q(v^*)).
\]
Let $w' = u \lvert w'Q \rvert$ be the polar
decomposition of $w'$. Since $\lvert w'Q \rvert$ is a compact perturbation of the identity, we have
that $u$, whose kernel is equal to $\text{ker}(w'Q) \subseteq \text{ker}(Q)$, is also a lift of $\phi(q(v^*))$
such that $u^* u \le Q$. Let $w$ be $u^*$. Summarizing, we can assume that both
$1 - w^*w$ and $Q - ww^*$ are finite rank projections.

In order to prove item \eqref{ic2}, first notice that the space
$K: = (1-Q) H \cap w^*w H$ is infinite dimensional, as $(1-Q) H$ is infinite dimensional and
$w^*wH$ has finite codimension. Let $n$ be the
rank of $1 - w^*w$ and fix a set of linearly independent vectors $\set{\zeta_k}_{k < n}$ in $K$.
Let $\set{\eta_k}_{k < n}$ be a basis of $(1 - w^*w) H$ and modify $w$ to be the operator sending
all vectors in $\set{\zeta_k}_{k < n}$ to zero, sending $\eta_k$ to $w(\zeta_k)$ for every $k < n$, and
acting as $w$ everywhere else. With these (compact) modifications, we have that $w^*w \ge Q$.
\end{proof}

Let $\set{\zeta_k}_{k < n}$ be an orthonormal basis of $(1 - w^*w) H$, and let 
$\set{\zeta_k}_{k \in \N}$ be an orthonormal basis of $H$ extending it. Denote the shift operator
sending $\zeta_k$ to $\zeta_{k+1}$ for all $k \in \N$ by $S$ and
let $r \in \mathbb{Z}$ be the difference
\[
r: = \text{rk}(Q - ww^*) - \text{rk}(1 - w^*w).
\]
The operator $S^{-r} S^r$ is either the identity (if $r$ is either zero or positive)
or a projection greater than $w^*w$, therefore, by Claim \ref{cc}
\begin{equation} \label{great}
S^{-r} S^r \ge w^*w \ge Q \ge ww^*.
\end{equation}
Consequently, $\tilde{w} := S^r w S^{-r}$ is a partial isometry such that
$\tilde{w}^* \tilde{w} = S^r w^* w S^{-r}$ and $\tilde{w} \tilde{w}^* = S^r ww^* S^{-r}$.
Moreover $\text{rk}(S^r(Q - ww^*)S^{-r}) = \text{rk}(Q - ww^*)$, since
$S^r$ acts as an isometry on $(Q - ww^*)H$ and $S^{-r}$ acts as an isometry on
$S^r(Q - ww^*)H$, which is the case since $S^{-r}S^r \ge Q \ge Q - ww^*$.

The operator $1 - S^r w^*w S^{-r}$ is the orthogonal projection onto the span of
$\set{\zeta_k}_{k < r+n}$, therefore $\text{rk}(1 - S^rw^*wS^{-r}) = \text{rk}(1 - w^*w) + r$,
thus
\[
\text{rk}(S^r(Q - ww^*)S^{-r}) = \text{rk}(1 - S^rw^*wS^{-r}).
\]
Because of this, there is a partial isometry $\overline{w}$ such that
\begin{equation} \label{w1}
\overline{w}^*\overline{w} = 1, \ \overline{w}\overline{w}^* = S^r Q S^{-r}
\end{equation}
and
\begin{equation} \label{w2}
\overline{w} \sim_{\cK(H)} S^r w S^{-r}.
\end{equation}
We claim that the map
\begin{align*}
\Phi_{\vec{E}} : \cD[\vec{E}] &\to \cB(H) \\
a &\mapsto \overline{w}^*S^r \Psi(vav^*)S^{-r} \overline{w}
\end{align*}
is a strongly continuous, unital $*$-homomorphism lifting $\text{Ad}(q(S^r)) \circ \phi$ on $\cD[\vec{E}]_\cQ$.
Strong continuity, linearity and preservation of the adjoint operation follow since $\Psi$ has these
properties.
Unitality is a consequence of the definition of $\overline{w}$:
\[
\overline{w}^*S^r \Psi(vv^*)S^{-r} \overline{w} = \overline{w}^*S^r QS^{-r} \overline{w}
\stackrel{\eqref{w1}}{=}
\overline{w}^* \overline{w} \overline{w}^* \overline{w} \stackrel{\eqref{w1}}{=} 1.
\]
Given $a,b \in \cD[\vec{E}]$ we have that
\begin{align*}
\overline{w}^*S^r \Psi(vabv^*)S^{-r} \overline{w} \ &=
\ \overline{w}^*S^r \Psi(vav^*Pvbv^*)S^{-r} \overline{w}  \\  &=
\ \overline{w}^*S^r \Psi(vav^*) Q \Phi(vbv^*)S^{-r} \overline{w}  \\ &\stackrel{\mathclap{\eqref{great}}}{=}
\ \overline{w}^*S^r \Psi(vav^*) S^{-r} S^r Q S^{-r} S^r \Psi(vbv^*)S^{-r} \overline{w}  \\ & \stackrel{\mathclap{\eqref{w1}}}{=}
\ \overline{w}^*S^r \Psi(vav^*) S^{-r} \overline{w}\overline{w}^* S^r \Psi(vbv^*)S^{-r} \overline{w}.
\end{align*}
Finally, for $a \in \cD[\vec{E}]$, the following holds
\begin{align*}
q(\overline{w}^*S^r \Psi(vav^*)S^{-r} \overline{w}) \ & \stackrel{\mathclap{\eqref{w2}}}{=}
\ q(S^r w^* S^{-r}S^r)\phi(q(vav^*))q(S^{-r}S^r w S^{-r})  \\ & \stackrel{\mathclap{\eqref{QP}}}{=}
\ q(S^r w^*) q( S^{-r}S^r Q)\phi(q(vav^*))q(Q S^{-r}S^r) q( w S^{-r})  \\  &\stackrel{\mathclap{\eqref{great}}}{=}
\ q(S^r w^*) q( Q)\phi(q(vav^*))q(Q ) q( w S^{-r}) \\ &\stackrel{\mathclap{\eqref{QP}}}{=}
\ q(S^r w^*) \phi(q(vav^*))q( w S^{-r})  \\ & \stackrel{\mathclap{\eqref{w}}}{=}
\ q(S^r) \phi(q(a)) q(S^{-r}).
\end{align*}
\end{proof}

\section{Classification and Closure Properties} \label{Sclass}
We are finally ready to prove Theorem \ref{mt2} and Theorem \ref{closure}.

\begin{proof}[Proof of Theorem \ref{mt2}]

The forward direction of the equivalence is straightforward.

Suppose thus that $\phi_1$, $\phi_2$ are two
endomorphisms of $\cQ(H)$ that satisfy conditions \eqref{mta} and \eqref{mtb}
of the statement. By condition \eqref{mta}
we can assume that $\phi_1(1) = \phi_2(1) = p$. Notice that if there is a unitary $v \in \cQ(H)$ such that
$\text{Ad}(v) \circ \phi_1 = \phi_2$, then $vpv^* = p$, hence $v$ and $p$ commute.
This means that $v$ is a direct sum of
a unitary in $p\cQ(H) p$ and a unitary in $(1-p) \cQ(H) (1-p)$.
The only part of $v$ acting non-trivially on $\phi_1[\cQ(H)]$ is the one in $p \cQ(H) p$.
Hence, without loss of generality,
we can assume that $p= \phi_1(1) = \phi_2(1) = 1$.

By Theorem \ref{thrm:main} both $\phi_1$ and $\phi_2$ are trivial.
By definition of trivial, unital endomorphism, for $i=1,2$
there is a unitary $u_i \in \cQ(H)$ such that $\text{Ad}(u_i) \circ \phi_i$
lifts to strongly continuous endomorphism $\Phi_i$ of
$\cB(H)$, which can assumed to be unital (see Remark \ref{remark:uni}).
This means that there exists
$m_i \in \N \setminus \set{0}$ such that, up to unitary transformation, $\phi_i$
lifts to the map $\Phi_{m_i}: \cB(H) \to \cB(H \otimes \C^{m_i})$ sending $T$ to $T \otimes 1_{m_i}$
(see Remark \ref{remark:sc}). In particular, the index of $\phi_i(q(S))$
is $-m_i$. By condition \eqref{mtb} we have that $m_1 = m_2$, hence $\phi_1$ and $\phi_2$ are
unitarily equivalent.

The final sentence of the theorem follows from the previous paragraph of the proof,
since $\Phi_m \oplus \Phi_n = \Phi_{m+n}$ and $\Phi_m \circ \Phi_n = \Phi_{mn}$ for
every $m, n \in \N \setminus \set{0}$.
\end{proof}
\begin{remark} \label{nonu}
Every non-unital $\phi \in \text{End}(\cQ(H))$ can be written as a direct sum
$\phi_1 \oplus 0$ where $\phi_1$ is unital and $0$ is the zero endomorphism of $\cQ(H)$.
Therefore, by Theorem \ref{thrm:main}, every non-unital endomorphism $\phi$ is, up to
unitary equivalence, equal to $\Phi_m \oplus 0$ for some $m \in \N$.
Consider the set $\mathcal{N} := (\N \times \set{0,1}) \setminus \set{(0,1)}$ and, for
$\phi \in \text{End}(\cQ(H))$,
define $\text{ind}(\phi) := -\text{ind}(\phi(q(S)) + (1 - \phi(1)))$, where $S$ is the unilateral shift.
Consider the map:
\begin{align*}
\Theta: \text{End}(\cQ(H)) &\to \mathcal{N} \\
\phi &\mapsto
\begin{cases}
(0,0) \text{ if } \phi(1) = 0  \\
 (\text{ind}(\phi),1) \text{ if } \phi(1) = 1 \\
(\text{ind}(\phi),0) \text{ if } 0 < \phi(1) < 1
\end{cases}
\end{align*}
By Theorem \ref{thrm:main} and the previous observation,
the map $\Theta$ is a bijection, since all projections $p \in \cQ(H)$ such
that $0 < p < 1$ are unitarily equivalent in $\cQ(H)$. For $\phi_1, \phi_2 \in \text{End}(\cQ(H))$
we have that $\phi_1 \oplus \phi_2$ (and $\phi_1 \circ \phi_2$) is non-unital if and only if at least one
between $\phi_1$ and $\phi_2$ is non-unital.
Therefore, the map $\Theta$ is a semigroup isomorphism
between $(\text{End}(\cQ(H)), \oplus)$ and $(\mathcal{N}, +)$, where
the addition on $\mathcal{N}$ is defined as $(n,i) + (m,j) = (n+m,i\cdot j)$.
Analogously, $\Theta$ is an isomorphism between $(\text{End}(\cQ(H)), \circ)$ and $(\mathcal{N}, \cdot)$, where $(n,i) \cdot (m,j) = (n\cdot m,i\cdot j)$.
\end{remark}

\begin{proof}[Proof of Theorem \ref{closure}]
\eqref{oca1} Let $\phi: \cQ(H) \to \cQ(H)$ be a trivial, unital endomorphism. Up to unitary transformation there is
$m \in \N$ such that $\phi$ is induced by the map $\Phi_m: \cB(H) \to \cB(H \otimes \C^m)$ sending
$T$ to $T \otimes 1_m$ (see Remark \ref{remark:sc}). The commutant of $\Phi_m[\cB(H)]$ in $\cB(H \otimes \C^m)$ is isomorphic to $M_m(\C)$. By \cite[Lemma 3.2]{jp} also the commutant of $\phi[\cQ(H)]$
in the codomain $\cQ(H)$ is isomorphic to $M_m(\C)$. Thus, the commutant of
the image of a trivial, unital endomorphism of $\cQ(H)$ is always finite dimensional.

Let
$\cB \in \mathbb{E}$ be unital, infinite-dimensional \cstar-algebra. Consider the algebraic tensor product
$\cQ(H) \otimes_{\text{alg}} \cB$ and suppose $\psi: \cQ(H) \otimes_{\text{alg}} \cB \to \cQ(H)$ is
an embedding. The element $\psi(1)$ is a non-zero projection. Since $\psi(1) \cQ(H) \psi(1) \linebreak
\cong \cQ(H)$, we can assume that $\psi$ is unital. On the one hand, by Theorem \ref{thrm:main}
 the restriction of $\psi$ to $\cQ(H)$ is unital and trivial, on the other hand $\psi$ sends injectively
 $\cB$ into the commutant
 of $\psi[\cQ(H)]$, which is finite-dimensional. This is a contradiction, since $\cB$ is assumed to be
 infinite-dimensional.

\eqref{oca2} Let $\set{\cA_n}_{n \in \N}$ be an increasing sequence of finite-dimensional, unital \cstar-algebra
such that $\cA := \overline{\bigcup_{n \in \N} \cA_n}$ is an infinite-dimensional, unital AF-algebra.
We have
that $\cQ(H) \otimes \cA_n \in \mathbb{E}$ for all $n \in \N$, but, by the proof of item \eqref{oca1}, $\cQ(H)
\otimes \cA \notin \mathbb{E}$.
\end{proof}

By the results in \cite{inner} we know that it is consistent with ZFC that there is
no automorphism of $\cQ(H)$ sending the unilateral shift $S$ to its adjoint.
We can generalize this statement to unital endomorphisms of $\cQ(H)$.
\begin{corollary} \label{corollary}
Assume OCA. There is no unital endomorphism $\phi: \cQ(H) \to \cQ(H)$ sending
the unilateral shift to its adjoint or to any unitary of index zero.
\end{corollary}
\begin{proof}
By Theorem \ref{thrm:main} all endomorphisms are trivial. In the proof of Theorem \ref{mt2} we showed that $\text{ind}(\phi(q(S)))$ is always a negative number if $\phi$ is unital and trivial. Since $\text{ind}(q(S^*)) =1$, the conclusion of the corollary follows.
\end{proof}

The following two examples witness the failure of Theorem \ref{mt2} and Theorem \ref{thrm:main} when
CH holds, suggesting a rather complicated picture of $\text{End}(\cQ(H))$ in that case.
In particular, while $\text{End}(\cQ(H))$ is countable under OCA (Theorem \ref{mt2}),
CH implies that $\text{End}(\cQ(H))$ has size $2^{\aleph_1}$, as shown in the following example.

\begin{example} \label{ex1}
An automorphism is trivial if and only if it is inner. Hence all the inequivalent $2^{\aleph_1}$ outer automorphisms
produced in \cite{outer} are examples of non-trivial endomorphisms of $\cQ(H)$ (outer
automorphisms of $\cQ(H)$ can also be built using the weakening of the
Continuum Hypothesis $\mathfrak{d} < \aleph_1 + 2^{\aleph_0} < 2^{\aleph_1}$,
see \cite{outer2} and \cite[\S17.1]{ilijasbook}).
All known cases of outer automorphisms of $\cQ(H)$ locally behave like inner automorphisms,
in particular they cannot change the value of the Fredholm index of a given element. Therefore, under CH
it is possible to find $2^{\aleph_1}$ inequivalent automorphisms which all send the unilateral
shift to an element of index -1.
\end{example}

\begin{example} \label{ex2}
In \cite{fhv} it is proved that all \cstar-algebras of density character
$\aleph_1$ can be embedded into $\cQ(H)$ with a map whose
restriction to a given separable, unital subalgebra is
a trivial extension\footnote{Given a \cstar-algebra
$\cA$ and a unital embedding $\theta: \cA \to \cQ(H)$, the map $\theta$ is a \emph{trivial extension of $\cA$} iff
there is a unital $*$-homomorphism $\Theta:\cA \to \cB(H)$ such that $\theta = q \circ \Theta$. This notion
of trivial maps, albeit more common, is different from the one we used throughout this paper in
in Definition \ref{trivial}.}. Under CH the
density character of $\cQ(H)$ is $\aleph_1$, thus there is a unital endomorphism
$\phi: \cQ(H) \to \cQ(H)$ that sends the unilateral shift $S$ to a unitary in $\cQ(H)$ which lifts to a unitary in $\cB(H)$,
which has therefore index zero.

By Corollary \ref{corollary}, this map cannot be a trivial endomorphism. With this example we see that without OCA the
range of values assumed by $\text{ind}(\phi(q(S)))$, as $\phi$ varies in $\text{End}_u(\cQ(H))$,
can be strictly larger than the negative numbers. Notice that the existence of an endomorphism
sending $S$ to $S^*$ would give an example where the index of the image of the shift is positive.
\end{example}

We do not know whether OCA is an optimal assumption to prove Theorem \ref{thrm:main2},
Theorem \ref{thrm:lctr} and their consequences exposed in the current section,
or if a weaker assumption might suffice.
Moreover, to the best of our knowledge, it is not known whether there are models of ZFC
where Theorem \ref{thrm:main2} holds and Theorem \ref{thrm:lctr} fails, or vice versa.

On the other hand, we can trace back
the failure of Theorem \ref{thrm:main} under CH to
the fact that both
Theorem \ref{thrm:main2} and Theorem \ref{thrm:lctr} fail under CH. Hence neither of them
can be proved in ZFC alone.

Indeed, the outer automorphisms built of $\cQ(H)$ in \cite[\S1]{inner} are locally trivial according
to our definition. They thus witness the failure of Theorem \ref{thrm:main2} under CH.

For what concerns Theorem \ref{thrm:lctr}, we have the following.
We have shown in the proof of Theorem \ref{mt2} that all unital trivial endomorphisms
send the unilateral shift to a unitary of $\cQ(H)$ with negative index. The same thing
can be proved for unital locally trivial endomorphisms.
Indeed, the unilateral shift $S$ is equal to $S_0 + S_1$ for
some $S_0 \in \cD[\vec E^{\text{id},\text{even}}]$, $S_1 \in \cD[\vec E^{\text{id},\text{odd}}]$.
Thus, by the results in \S\ref{S3} and Lemma \ref{lemma:double},
if $\phi: \cQ(H) \to \cQ(H)$ is a unital, locally trivial endomorphism,
it follows that $\text{ind}(\phi(q(S))) = -m$ for some non-zero $m \in \N$.
Therefore, the endomorphism provided in Example \ref{ex1} cannot be locally trivial.

In \cite{inner2} the author extends his results in \cite{inner} to all Calkin algebras on
nonseparable Hilbert spaces, showing that the Proper Forcing Axiom implies that
all automorphisms of the Calkin algebra on a nonseparable Hilbert space are inner. It would
be interesting to known whether those techniques can be generalized to study
the semigroup of endomorphisms of the Calkin algebra on a nonseparable Hilbert space.

In conclusion, we remark that an
interesting consequence of the simple structure of $\text{End}_u(\cQ(H))$ under OCA
is that the monoid $(\text{End}_u(\cQ(H)), \circ)$ is commutative (Theorem \ref{mt2}).
We expected this to be false under CH.

\begin{acknow}
I wish to thank Ilijas Farah for his suggestions concerning these problems and for his useful
remarks on the early (and the final) drafts of this paper. I also would like to thank Alessandro Vignati for the valuable conversations we had about these topics. Finally, I thank the anonymous referee for their
helpful remarks on the paper.
\end{acknow}

\bibliographystyle{amsalpha}
	\bibliography{Bibliography}

\end{document}